\documentclass[final]{ectaart}
\usepackage{latexsym,amsfonts,amsmath,amssymb,amsthm,url,graphicx,psfrag,epsfig,natbib}%
\RequirePackage{hypernat}
\RequirePackage[colorlinks,citecolor=blue,urlcolor=blue]{hyperref}
\usepackage[latin1]{inputenc}
\usepackage[english]{babel}
\startlocaldefs
\numberwithin{equation}{section}
\theoremstyle{plain}
\newtheorem{thm}{Theorem}[section]
\newtheorem{coro}[thm]{Corollary}
\newtheorem{prop}[thm]{Proposition}
\newtheorem{lem}[thm]{Lemma}
\newtheorem{defi}[thm]{Definition}
\theoremstyle{definition}

\theoremstyle{remark}
\newtheorem{rem}[thm]{Remark}

\newcommand{\R}{\mathbb{R}}

\newcommand{\W}{\mathcal{W}}
\newcommand{\J}{\mathcal{J}}
\newcommand{\E}{\mathcal{E}}
\newcommand{\V}{\mathcal{V}}
\newcommand{\D}{\mathcal{D}}
\renewcommand{\L}{{\cal L}}

\newcommand\loc{{\mathrm{loc}}}
\newcommand\qtext[1]{\quad\text{#1}\quad}

\let \eps\varepsilon

\newcommand{\C}{\mathcal {C}}

\newcommand\CC{{\cal C}}

\newcommand\clos{\overline}

\DeclareMathOperator{\argmin}{argmin}

\DeclareMathOperator{\lip}{Lip}

\DeclareMathOperator{\diam}{diam}

\newcommand\dive{\mathrm{div}}

\def\<#1,#2>{\left<#1,#2\right>}

\DeclareMathOperator{\Lip}{Lip}

\def\P{{\cal P}}
\newcommand{\dd}{\;{\rm d}}

\newcommand{\id}{{\rm id}}

\newcommand{\un}{{\rm 1\kern -2.5pt l}}

\endlocaldefs
\begin{document}
%%%%%%%%%%%%%%%%%%%%%%%%%%%%%%%
%%%%%%%%%%%%%%%%%%%%%%%%%%%%%%% 
\begin{frontmatter}
\title{Optimal transport and Cournot-Nash equilibria}
\runtitle{Optimal transport and Cournot-Nash equilibria}
\begin{aug}
\author{\fnms{Adrien} \snm{Blanchet}\ead[label=e1]{Adrien.Blanchet@univ-tlse1.fr}\thanksref{t1}}
 \thankstext{t1}{TSE (GREMAQ, Universit\'e de Toulouse), Manufacture des Tabacs, 21 all\'ee de Brienne, 31000 Toulouse, FRANCE. \url{Adrien.Blanchet@univ-tlse1.fr}} 
%\address{TSE (GREMAQ, Universit\'e de Toulouse), Manufacture des Tabacs, 21 all\'ee de Brienne, 31000 Toulouse, FRANCE. \url{Adrien.Blanchet@univ-tlse1.fr}}
%\and
\author{\fnms{Guillaume} \snm{Carlier}\ead[label=e2]{carlier@ceremade.dauphine.fr}\thanksref{t2}}
\thankstext{t2}{CEREMADE, UMR CNRS 7534, Universit\'e Paris-Dauphine,  Pl. de Lattre de Tassigny, 75775 Paris Cedex 16, FRANCE. \url{carlier@ceremade.dauphine.fr}}
%\address{CEREMADE, UMR CNRS 7534, Universit\'e Paris-Dauphine,  Pl. de Lattre de Tassigny, 75775 Paris Cedex 16, FRANCE. \printead{e2}}
\end{aug}
\runauthor{A. Blanchet \& G. Carlier}
%%%%%%%%%%%%%%%%%%%%%%%%%%%%
\begin{abstract}
~ We study a class of games with a continuum of players for which Cournot-Nash equilibria can be obtained by the minimisation of some cost, related to optimal transport.  This cost is not convex in the usual sense in general but it turns out to have hidden strict convexity properties in many relevant cases. This enables us to obtain  new uniqueness results and a characterisation of equilibria in terms of some partial differential equations, a simple numerical scheme in dimension one as well as an analysis of the inefficiency of equilibria. 
\end{abstract}
%%%%%%%%%%%%%%%%%%%%%%%%%%%%
\begin{keyword}
\kwd{Cournot-Nash equilibria}
\kwd{mean-field games}
\kwd{optimal transport}
\kwd{externalities}
\kwd{Monge-Amp\`ere equations}
\kwd{convexity along generalised geodesics}
\end{keyword}
\end{frontmatter}
%%%%%%%%%%%%%%%%%%%%%%%%%%%%%%%
%%%%%%%%%%%%%%%%%%%%%%%%%%%%%%% 
%%%%%%%%%%%%%%%%%%%%%%%%%%%%%%%
\section{Introduction}
Since Aumann's seminal works~\cite{Aumann,Aumann2}, models with a continuum of agents have occupied a distinguished position in economics and game theory.  \cite{Schmeidler} introduced a notion of non-cooperative equilibrium in games with a continuum of agents and established several existence results. In Schmeidler's own words: \emph{Non-atomic games enable us to analyze a conflict situation where the
single player has no influence on the situation but the aggregative behavior
of "large" sets of players can change the payoffs. The examples are numerous:
Elections, many small buyers from a few competing firms, drivers that can
choose among several roads, and so on.}

Following the approach of~\cite{hhk} to Walras equilibrium analysis, \cite{MasColell} reformulated Schmeidler's analysis in terms of joint distributions over agents' actions and characteristics and, in particular, the concept of Cournot-Nash equilibrium distributions. Not only  Mas-Colell's reformulation enabled him to obtain general existence results in an easy and elegant way but it is flexible enough to accommodate quite weak assumptions on the data (which is relevant in the framework of games with incomplete information and a continuum of players for instance). Roughly speaking, in analysing  Cournot-Nash equilibria in the sense of~\cite{MasColell} one can take great advantage of (topological but also geometric) properties of spaces of probability measures. With this respect, it is natural to expect that optimal transport theory (which is an extremely active field in research in mathematics both from an applied and fundamental point, as illustrated by the monumental textbook~\cite{villani2}) may be useful. 

Even though there are very general existence results for Cournot-Nash equilibria (see for instance~\cite{Kahn}) in the literature, we are not aware of classes of problems where there is uniqueness and a full characterisation of such equilibria which is tractable enough to obtain close-form solutions or efficient numerical computation schemes. One of our goals is precisely to go one step beyond abstract existence results (in mixed or pure strategies) and to identify classes of non-atomic games where  Cournot-Nash equilibria are unique and can be fully characterised or numerically computed.

\smallskip

Given a space of players types $X$ endowed with a probability measure $\mu\in \P(X)$  (which gives the exogenous distribution of the type of the agents), an action space $Y$ and a cost $\Phi$: $X\times Y\times \P(Y) \to \R$, $x$-type agents taking action $y$ pay the cost $\Phi(x,y, \nu)$ where $\nu\in \P(Y)$ represents the action distribution. The fact that this cost depends on the other agents actions only through the distribution $\nu$ means that who plays what does not matter {\it i.e.} the game is \emph{anonymous}. A \emph{Cournot-Nash equilibrium} is a joint probability measure $\gamma\in \P(X\times Y)$ with first marginal $\mu$  such that 
\begin{equation}\label{equiggmc}
\gamma(\{(x,y)\in X\times Y\; : \; \Phi(x,y, \nu)=\min_{z\in Y} \Phi(x,z, \nu)\})=1
\end{equation}
where $\nu$ represents  $\gamma$'s second marginal. The probability $\gamma$ is naturally interpreted by saying that $\gamma(A\times B)$ is the probability that agents have their type in $A$ and an action in $B$. The equilibrium is called \emph{pure} if, in addition, $\gamma$ is carried by a graph {\it i.e.} $\mu$-a.e. the agents play in pure strategy. Condition~\eqref{equiggmc} means that agents choose cost minimising strategies given their type and $\nu$ so that, finally, imposing that $\nu$ is the second marginal of $\gamma$ is a simple self-consistency requirement.\smallskip

In the sequel, we will restrict ourselves to the additively separable case where $\Phi(x,y, \nu)=c(x,y)+\V[\nu](y)$, which seems to be a necessary limitation for the optimal transport approach we will develop. Under this separability specification, the connection with optimal transport is almost obvious: if $\gamma$ is a Cournot-Nash equilibrium it necessarily minimises the average of $c$ among probability measures having $\mu$  and  $\nu$ (which is {\it a priori} unknown) as marginals {\it i.e.} it solves the optimal transport problem:
\begin{equation}\label{mkwcintro}
\W_c(\mu, \nu):=\inf_{\gamma\in \Pi(\mu, \nu)} \iint_{X\times Y} c(x,y) \dd\gamma(x,y)
\end{equation}
where $\Pi(\mu, \nu)$ is the set of joint probabilities having $\mu$ and $\nu$ as marginals.  In an euclidean setting, there are well-known conditions on $c$ and $\mu$ which guarantee that such an optimal $\gamma$ necessarily is pure whatever $\nu$ is and this  of course implies purity of equilibria. 

If we go one step further and assume  that $\V[\nu]$ is the differential of some functional $\E$ (see Section~\ref{varapp} for a precise definition), it turns out that if $\nu$ is a  minimiser of $\E[\nu]+\W_c(\mu,\nu)$ and $\gamma$ solves~\eqref{mkwcintro} then it is a Cournot-Nash equilibrium.  This gives a variational device to find equilibria: first find $\nu$ by minimising  $\E[\nu]+\W_c(\mu,\nu)$ and then find $\gamma$ by solving the optimal transport problem~\eqref{mkwcintro} between $\mu$ and $\nu$. This variational approach actually gives new existence results. To the best of our knowledge, usual general existence proofs are via fixed-point arguments and thus require a lot of regularity for the dependence of $\V[\nu]$ with respect to $\nu$, in a compact  metric setting, it is typically asked that $\nu\mapsto \V[\nu]$ is  continuous (or at least upper-semi continuous in some sense) from the set of probabilities equipped with the weak-$*$ topology to the set of continuous functions equipped with the supremum norm. This is harmless if $Y$ is finite but extremely restrictive in general, in particular it excludes the case of a purely local dependence which is relevant to capture congestion effects (actions that are frequently played are more costly). In contrast, the variational approach will enable us to treat such local congestion effects.  If $\E$ (the primitive of $\V$ in some sense) is convex then equilibria and minimisers coincide and strict convexity gives uniqueness of the second marginal, $\nu$, of the equilibrium. Such a convexity is quite demanding in applications but we shall prove that in an euclidean setting and for a quadratic $c$ (and more generally strictly convex $c$'s in dimension one), there is some hidden convexity  (in the spirit of the seminal results of~\cite{McC2})
in the problem from which one can deduce uniqueness of equilibria but also a characterisation in terms of a nonlinear partial differential equation of Monge-Amp\`ere type. This partial differential equation cannot be solved explicitly in general but, in dimension one, it is easy to solve the variational problem numerically in an efficient way as we shall illustrate on several examples. Another advantage of the variational approach is that it allows for an elementary (in-)efficiency analysis of the equilibrium and the design of a tax system to restore the efficiency of the equilibrium (see Section~\ref{concl}). Of course, the variational approach described above presents strong similarities with the potential games of \cite{MS} and our framework is very close to that of \cite{klw} or \cite{lw1} in the case of a finite number of players; however we are not aware of any extension of the analysis of \cite{MS} to the case of a continuum of players.

Apart from our results on Cournot-Nash equilibria, another objective of the paper is to contribute to popularise the use of optimal transport in economics.  Several recent papers have fruitfully used optimal transport arguments in such different fields as hedonics and matching problems (\cite{Ekelandhedonic,chiapp}), multidimensional screening  (\cite{gcmonge,figalli}) or urban economics (\cite{gcie,BMS}). We believe that cross-fertilisation between economics and optimal transport  will rapidly develop.  This is why  we have included in Appendix~\ref{tool} some basic results from optimal transport theory which we hope can serve as a comprehensive introduction to this vast subject to an economists readership. 

The present introduction would neither be complete nor fair without an explicit reference to the mean-field games theory of~\cite{mfg1,mfg2,mfg3}. Indeed, our variational approach is largely inspired by the Lasry and Lions optimal control approach to mean-field games (that has some similarities with optimal transport), but also mean-field games theory enables to treat considerably richer situations than the somehow static one we treat here. Another line of research we would like to mention concerns congestion games (another example of potential games) and the literature on the cost of anarchy (see~\cite{rough} and the references therein), indeed the variational approach we develop presents some similarities with the variational approach to Wardrop equilibria on congested networks and in both cases equilibria are socially inefficient.  \smallskip

The paper is organised as follows. In Section~\ref{secmodel}, we introduce the model, define equilibria and emphasise some connections with optimal transport. In Section~\ref{varapp}, we adopt a  variational approach and prove that for a large class of interactions, equilibria naturally arise as local minimisers of a certain functional. Section~\ref{hidd} is devoted to further uniqueness and variational characterisation of equilibria  results thanks to notions of displacement convexity arising in optimal transport, we also characterise the equilibrium via a certain nonlinear partial differential equation and compute numerically the equilibrium in dimension one. Section \ref{concl} concludes. The proofs as well as well as a presentation of various results from optimal transport theory which are used throughout the paper are gathered in the Appendix.

%%%%%%%%%%%%%%%%%%%%%%%
\section{The equilibrium model}\label{secmodel}
%%%%%%%%%%%%%%%%%%%%%%%

The model consists of a compact metric type space $X$ equipped with a Borel probability measure $\mu\in \P(X)$, giving the distribution of types, a compact metric action space $Y$, a reference\footnote{The role of the reference  measure $m_0$ is here to capture purely local congestion effects as in the examples below. In other words, we will require action distributions to be absolutely continuous with respect to $m_0$. This departs from the common assumption that the cost is well defined for every action distribution and satisfies some strong continuity/semi-continuity  with respect to the weak $*$ topology of measures as in \cite{MasColell} or \cite{Kahn}.}    Borel non-negative measure $m_0$, a continuous function $c\in \C(X\times Y)$ and interactions are captured by a map which to every action distribution $\nu\in \P(Y)\cap \L^1(m_0)$ associates a function $\V[\nu]$ defined $m_0$-almost-everywhere. Given an action distribution $\nu$,  $x$-type agents taking action $y$ then incur the additively separable cost $c(x,y)+\V[\nu](y)$. The unknown is a probability distribution $\gamma\in \P(X\times Y)$, with the interpretation that $\gamma(A\times B)$ is the probability that an agent has her type in $A$ and takes an action in $Y$, such a $\gamma$ induces as action distribution $\nu$, its second marginal which we denote $\nu={\pi_Y}_\#\gamma$. By construction, the first marginal of $\gamma$, ${\pi_X}_\#\gamma$ should  be equal to $\mu$.  Since we will be interested by efficiency (or rather inefficiency) properties of equilibria, we will also impose that $\gamma$ has finite social cost, where the latter is given by
\begin{align*}
 {\rm{SC}}&=\iint_{X\times Y}  (c(x,y) +\V[\nu](y))\dd\gamma(x,y)\\
&=\iint_{X\times Y} c(x,y)\dd\gamma(x,y)+\int_Y \V[\nu](y) \dd\nu(y)
\end{align*}
since $c$ is continuous the first term is finite for every $\gamma$, but the second requires the action marginal $\nu$ to belong to the domain
\begin{equation}\label{domainD}
\D:=\{\nu \in \L^1(m_0)\; : \; \V[\nu] \in \L^1(\nu)\}=\{\nu \in \L^1(m_0)\; : \; \int_Y \vert \V[\nu]\vert \dd\nu<+\infty\}.
\end{equation}

Cournot-Nash equilibria are then defined as follows.
\begin{defi}
$\gamma\in \P(X\times Y)$ is a \emph{Cournot-Nash equilibria} if its first marginal is $\mu$, its second marginal, $\nu$, belongs to $\D$ and there exists $\varphi\in \C(X)$ such that 
\begin{equation}\label{equipp}
c(x,y)+\V[\nu](y)\geq \varphi(x) \mbox{ for all $x\in X$ and $m_0$-a.e. $y$ with equality $\gamma$-a.e.}.
\end{equation}
A Cournot-Nash equilibrium $\gamma$ is called \emph{pure} whenever it is carried by a graph {\it i.e.} is  of the form $\gamma=({\rm{id}} ,T)_{\#}\mu$ for some Borel map $T$ : $X\to Y$. 
\end{defi}
The previous definition is slightly different from that of \cite{MasColell} because we require the action distribution to be absolutely continuous with respect to $m_0$, so as to take into account congestion effects as explained in the examples below. This makes the existence of equilibria nontrivial, indeed, when $\nu \mapsto \V[\nu]$ is continuous  from $(\P(Y), w-$*$)$ to $(\C(Y), \Vert . \Vert_{\infty})$ (as is the case for instance when $\V[\nu](y)=\int_Y \phi(y,z) \dd\nu(z)$ with $\phi$ continuous) standard fixed-point arguments immediately give the existence of Cournot-Nash equilibria but here, we do not have such regularity.

\subsection{Examples}

\subsubsection*{Holiday choice}
Let us consider a population of agents whose location is distributed according to some probability distribution $\mu\in \P(X)$ where $X$ is some compact subset of $\R^2$ (say). These agents have to choose their holidays destination (possibly in mixed strategy). The set of possible holiday destinations is some compact subset of the plane $Y$ (it can be $X$, a finite set, ...). The commuting cost from $x$ to $y$ is $c(x,y)$. In addition to the commuting cost, agents incur costs resulting from interactions with other agents, this is captured by a map $\nu\mapsto \V[\nu]$ that can be modelled as follows. A natural effect that has to be taken into account is congestion, {\it i.e.} the fact that more crowded location results in more disutility for the agents. Congestion thus requires to consider local effects and actually imposes that $\nu$ is not too concentrated; a way to capture this is to impose that $\nu$ is absolutely continuous with respect to some reference probability measure $m_0$. Still denoting by $\nu$ the Radon-Nikodym derivative of $\nu$, a natural congestion cost is of the form $y\mapsto f(\nu(y))$ with $f$ non-decreasing. In addition to the negative externality due to congestion effect, there may be a positive externality effect due to the positive social interactions between agents  which can be captured through a non-local term of the form $y\mapsto \int_Y \phi(y,z) \dd\nu(z)$ where for instance $\phi(y,.)$ is minimal for $z=y$ so that the previous term represents a cost for being far from the rest of the population. 
Finally, the presence of purely geographical factors (e.g.  distance to the sea) can be reflected by a term of the form $y \mapsto v(y)$. The total externality cost generated by the distribution $\nu$ combines the three effects of congestion, positive interactions and geographical factors and can then be taken of the form 
\[\V[\nu](y)= f( \nu(y))+\int_Y \phi(y,z) \dd\nu(z) + v(y).\]

\subsubsection*{Technological choice}

Consider now a simple model of technological choice in the presence of externalities. There is a set of consumers indexed by a type $x\in X$ drawn according to the probability $\mu$, and a set of technologies $Y$ for a certain good (cell-phone, computer, tablet...). On the supply side, assume there is a single profit maximising profit firm with convex production cost $F(y,.)$ producing technology $y$, the supply (equals demand at equilibrium) of this firm is thus determined by the marginal pricing rule $p(y)=\partial_\nu F(y, \nu(y))$.  Agents aim to minimise with respect to $y$ a total cost which is the sum of their individual purchasing cost $c(x,y)+p(y)=c(x,y)+\partial_\nu F(y, \nu(y))$ and an additional usage/maintenance or accessibility cost which is positively affected by the number of consumers having purchased similar technologies {\it i.e.} a term of the form $\int_Y \phi(y,z) \dd\nu(z)$ where $\phi$ is increasing in the distance between technologies $y$ and $z$.

\subsection{Connection with optimal transport and purity of equilibria}

For $\nu\in \P(Y)$, let $\Pi(\mu, \nu)$ denote the set of probability measures on $X\times Y$ having $\mu$ and $\nu$ as marginals and let $\W_c(\mu, \nu)$ be the least cost of transporting $\mu$ to $\nu$ for the cost $c$ {\it i.e.} the value of the Monge-Kantorovich optimal transport problem:
\[\W_c(\mu, \nu):=\inf_{\gamma\in \Pi(\mu, \nu)}  \iint_{X\times Y} c(x,y) \dd\gamma(x,y)\]
let us also denote by $\Pi_o(\mu,\nu)$ the (nonempty) set of optimal transport plans {\it i.e.}
\[\Pi_o(\mu,\nu):=\{\gamma\in \Pi(\mu, \nu) \; : \; \iint_{X\times Y} c(x,y) \dd\gamma(x,y)=\W_c(\mu, \nu)\}.\]

A first link between Cournot-Nash equilibria and optimal transport is based on the following straightforward observation.

\begin{lem}
 If $\gamma$ is a Cournot-Nash equilibrium and $\nu$ denotes its second marginal then $\gamma\in \Pi_o(\mu, \nu)$.
 \end{lem}
 \begin{proof}
  Indeed, let $\varphi\in \C(X)$ be such that~\eqref{equipp} holds and let $\eta\in \Pi(\mu, \nu)$ then we have 
\begin{multline*}
 \iint_{X\times Y} c(x,y) \dd\eta(x,y)\geq  \iint_{X\times Y} (\varphi(x)-\V[\nu](y)) \dd\eta(x,y)\\
 =\int_X \varphi(x) \dd\mu(x)-\int_Y \V[\nu](y) \dd\nu(y)
 = \iint_{X\times Y} c(x,y) \dd\gamma(x,y)
 \end{multline*} 
so that $\gamma\in \Pi_o(\mu, \nu)$.
\end{proof} 
The previous proof also shows that $\varphi$ solves the dual of $\W_c(\mu, \nu)$ (see Appendix~\eqref{dualkanto}) {\it i.e.} maximises the functional
\[\int_X \varphi(x) \dd\mu(x) +\int_Y \varphi^c(y) \dd\nu(y)\]
where $\varphi^c$ denotes the $c$-transform of $\varphi$ {\it i.e.}
\begin{equation}\label{ctransffi}
\varphi^c(y):=\min_{x\in X} \{c(x,y)-\varphi(x)\}\;.
\end{equation}

 In an euclidean setting, there are well-known conditions on $c$ and $\mu$ which guarantee that such an optimal $\gamma$ necessarily is pure whatever $\nu$ is. It is the case for instance if $\mu$ is absolutely continuous with respect to the Lebesgue measure, $c(x,y)$ is a smooth and strictly convex function of $x-y$ (see  \cite{mcgan} who extended the seminal results of \cite{bre} in the quadratic cost case), or more generally, when it satisfies a generalised Spence-Mirrlees condition (see \cite{gcmonge} for details):

\begin{coro}\label{purityofCNE}
Assume that $X=\clos{\Omega}$ where $\Omega$ is some open connected bounded subset of $\R^d$ with negligible boundary, that $\mu$ is absolutely continuous with respect to the Lebesgue measure, that $c$ is differentiable with respect to its first argument, that $\nabla_x c$ is continuous on $\R^d\times Y$ and that it satisfies the generalised Spence-Mirrlees condition: 
\[\mbox{for every $x\in X$, the map $y\in Y\mapsto \nabla_x c(x,y)$ is injective,}\]
then for every $\nu\in \P(Y)$, $\Pi_o(\mu, \nu)$ consists of a single element and the latter is of the form $\gamma=({\rm{id}}, T)_\#\mu$ hence every Cournot-Nash equilibrium is pure. 
\end{coro}

%%%%%%%%%%%%%%%%%%%%%%%%%%%%%
%%%%%%%%%%%%%%%%%%%%%%%%%%%%%

\section{A variational approach}\label{varapp}
In this section, we will see that in many relevant cases, one may obtain equilibria by the minimisation of some functional over a set of probability measures\footnote{Note the analogy with the variational approach of \cite{MS} for potential games, {\it i.e.} games whose equilibria can be obtained by minimising some potential function.}. The main assumption for this variational approach to be valid is that the interaction map $\V[\nu]$ has the structure of a differential {\it i.e.} that $\V[\nu]$ can be seen as the first variation of some function $\nu\mapsto \E[\nu]$. In this case, the variational  approach is based on the observation that the equilibrium condition is the first-order optimality condition for the minimisation of $\W_c(\mu, \nu)+\E[\nu]$.
%%%%%%%%%%%%%%%%%%%%%%%%%%%%%
\subsection{Interaction maps which are differentials}
The main assumption for the variational approach to be valid is that $\nu\mapsto \V[\nu]$ is a differential in the following sense:
%---------------------------------
\begin{defi}[Differential]\label{vhasapot}
Let $\D$ be defined by~\eqref{domainD}.  The map $\nu\in  \D\mapsto \V[\nu]$ is a \emph{differential} on $\D$ if $\D$ is convex and there exists $\E$: $\D \to \R$ such that for every $(\rho, \nu)\in  \D^2$, $\V[\nu]\in \L^1(\rho)$ and 
\[\lim_{\eps\to 0^+} \frac{\E[(1-\eps)\nu+\eps\rho]-\E[\nu]}{\eps}=\int_{Y} \V[\nu] \dd(\rho-\nu)\]
{\it i.e.} $\V[\nu]$ is the first variation of $\E$ which we denote $\V[\nu] =\dfrac{\delta \E}{\delta \nu}$. 
\end{defi}
%---------------------------------

Before going any further, let us consider some examples to illustrate the previous definition. 
%%%%%%%%%%%%%%%%%%%
\subsubsection*{Local term}
Let us consider first the case of a local dependence, again $m_0$ is our reference measure and for $\nu\in \D:=\P(Y)\cap \L^1(m_0)$ and $m_0$-a.e. $y$:
\[\V[\nu](y)=f(y, \nu(y))\]
for some continuous $f$. Assume first that $f$ is bounded and define, for all $\nu\in \D$:
\begin{equation}\label{primit}
F(y, \nu)=\int_0^\nu f(y, s)\dd s, \quad \E[\nu]=\int_Y F(y, \nu(y)) \dd m_0(y)\;.
\end{equation}
Then since $F$ is Lipschitz in $\nu$ uniformly in $y$, it easily follows from Lebesgue's dominated convergence theorem that  $\V[\nu]$ is the differential of $\E$ on $\D$. Now, rather assume that  $f$ satisfies the growth condition
\begin{equation}\label{growth1}
a(\nu^{\alpha} -1)\leq  f(y, \nu) \leq b(\nu^{\alpha}+1)
\end{equation}
for some $a\geq 0$, $b>0$, $\alpha>0$, $m_0$-a.e. $y$ and every $\nu\geq 0$. For $p=\alpha+1$, the corresponding energy functional $\E$ is then defined for all $\nu\in \D:=\P(Y)\cap \L^p(m_0)$ as above by~\eqref{primit}. Thanks to~\eqref{growth1} and Lebesgue's dominated convergence theorem, $\V$ is the differential of $\E$ on $\D$,  $\V[\nu]\in \L^{p'}(m_0)$ (with $p'$ the conjugate exponent of $p$ {\it i.e.} $p'=p/(p-1)$) as soon as $\nu\in \D$. Apart from the technical growth condition (which is useful to apply Lebesgue's theorem and guarantee that $\V[\nu] \nu$ is integrable) we therefore see that local $\V$'s are differentials. In Section~\ref{inada}, we will treat local $\V$'s under a different Inada-like condition on $f$ which is more customary in economics and will ensure that $\nu$ remains positive hence simplifying the equilibrium/optimality condition. 
%%%%%%%%%%%%%%%%%%%%%%%%%%%
\subsubsection*{Non-local interaction term and the role of symmetry}
Let us now consider the  case of (pairwise) interactions where $\V[\nu]$ is defined  by 
\[\V[\nu](y)=\int_Y \phi(y,z) \dd\nu(z)\]
for some $\phi\in \C(Y\times Y)$. It is then natural to define the quadratic functional
\[\E[\nu]=\frac{1}{2} \iint_{Y\times Y} \phi(y,z) \dd\nu(y) \dd\nu(z)\;.\]
By expanding in $\eps$, $\E[\nu+\eps(\rho-\nu)]$, its differential is immediate to compute
\begin{align*}
  \lim_{\eps \to 0} &\frac{\E[\nu +\eps(\rho-\nu)]-\E[\nu]}{\eps}\\
&\qquad=\frac12\iint \phi(y,z) [\!\dd\nu(y) \dd(\rho-\nu)(z) + \dd\nu(z) \dd(\rho-\nu)(y)]\\
&\qquad=\frac12\iint [\phi(y,z) + \phi(z,y)]  \dd\nu(z) \dd(\rho-\nu)(y)\;.
\end{align*}
So that
\[\frac{\delta \E}{\delta \nu}(y)=\int_Y \phi^{\rm{sym}}(y,z) \dd\nu(z) \; : \; \phi^{\rm{sym}}(y,z)=\frac{\phi(y,z)+ \phi(z,y)}{2}.\]
Hence $\V$ is the differential of $\E$ on $\P(Y)$ as soon as $\phi$ is symmetric\footnote{Let us remark that in the case of a finite number of players, the role of symmetry for the potential approach to work was already pointed out in \cite{lw1}.} {\it i.e.} $\phi(y,z)=\phi(z,y)$ (which is the case for instance if $\phi$ is the function of the distance between $y$ and $z$). Note that the assumption that $\V$ is a differential requires $\phi$ to be symmetric.\footnote{In a similar way, if we consider the case of higher-order interactions  
\[
\V[\nu](y)=\int_{Y^m} \phi(y,.) \dd\nu^{\otimes m}=\int_{Y^m} \phi(y,z_1,\ldots, z_m)\dd\nu(z_1) \ldots \dd\nu(z_m)
\]
where  $\phi\in \CC(Y^{m+1})$ satisfies the symmetry relations 
\begin{equation}\label{symm}
\phi(y,z_1,\ldots, z_m) =\phi(z_1,y,\ldots, z_m)=\cdots=\phi(z_m, z_1,\ldots, y),  
\end{equation}
for all $(y,z_1,\ldots, z_m)\in Y^{m+1}$, then $\V$ is the differential of 
\begin{align*}
 \E[\nu]&=\frac{1}{m+1}\int_{Y^{m+1}} \phi   \dd\nu^{\otimes (m+1)}. 
\end{align*}}
 Of course, one can combine the previous examples and consider a $\V$ which is the sum of a symmetric interaction term and a local term, such $\V$'s still have the structure of a differential. 
%%%%%%%%%%%%%%%%%%%%%%%%
\subsection{Minimisers are equilibria}
Throughout this paragraph, we assume that
\begin{equation}\label{vgrad}
\V[\nu] =\frac{\delta \E}{\delta \nu} \mbox{ on $\D$}.
\end{equation}
We then consider the variational problem 
\begin{equation}\label{varequil}
\inf_{\nu\in \D} \J_\mu[\nu]\quad\mbox{where}\quad \J_\mu[\nu]:=\W_c(\mu, \nu)+\E[\nu].
\end{equation}
To prove that minimisers of~\eqref{varequil} are equilibria, we first need to be able to differentiate the term $\W_c(\mu, \nu)$ with respect to $\nu$, this is possible thanks to Lemma~\ref{diffkanto} proved in Appendix~\ref{tool} but it requires more structure on $X$, $\mu$ and $c$: in particular $X$ is a connected subset of $\R^d$, $c$ is differentiable with respect to $x$ and $\mu$ is equivalent to the Lebesgue measure.\smallskip
%----------------------------
\begin{thm}[Minimisers are equilibria]\label{miniareeq}
Assume that $\V[\nu]$ satisfies~\eqref{vgrad}  with $\D=\P(Y)\cap \L^p(m_0)$ for some $p\in[1,+\infty[$ and that the assumptions of {\rm Lemma~\ref{diffkanto}} hold true. If
 $\nu$ solves~\eqref{varequil} and $\gamma \in \Pi_o(\mu,\nu)$ then $\gamma$ is a Cournot-Nash equilibrium.
\end{thm}
See Appendix~\ref{aminiareeq} for the proof. Let us mention however that the optimality condition for ~\eqref{varequil} is the following: there is a constant $M$ such that
\begin{equation}\label{eq:optcond}
\left\{
  \begin{array}{ll}
    \varphi^c+\V[\nu] \geq M\quad&\mbox{}\vspace{.1cm}\\
\varphi^c+\V[\nu]= M\quad&\mbox{$\nu$-a.e.}\;,
  \end{array}
\right.
\end{equation}
where $\varphi^c$ is the $c$-transform of $\varphi$ as in~\eqref{ctransffi}.

%----------------------------

To deduce an existence result from Theorem \ref{varequil}, assume that $\V[\nu]$ is defined for $\nu\in \P(Y)\cap \L^1(m_0)$ by
\begin{equation}
  \label{eq:vou}
  \V[\nu](y):=f(y, \nu(y))+ \int_{Y} \phi(y,z) \dd\nu(z)
\end{equation}
where $\phi\in \C(Y\times Y)$ is symmetric,  $f$ is continuous and non-decreasing with respect to its second argument and satisfies the growth condition~\eqref{growth1}
%a(\nu^{\alpha} -1)\leq  f(y, \nu) \leq b(\nu^{\alpha}+1)\quad \mbox{$m_0$-a.e. $y$ and every $\nu\geq 0$}
%\end{equation}
for some $a>0$, $b>0$ and $\alpha>0$. For $p=\alpha+1$, the corresponding energy functional is then defined for all $\nu\in \D:=\P(Y)\cap \L^p(m_0)$ by
\[\E[\nu]=\int_{Y} F(y, \nu(y))\dd m_0(y)+ \frac{1}{2}\iint_{Y^2} \phi(y,z) \dd\nu(y) \dd\nu(z)\]
where $F$ is defined by~\eqref{primit}. The functional $F(y,.)$ is convex and satisfies the growth condition
\begin{equation*}a(p^{-1}\nu^p-\nu)\leq F(y, \nu)\leq b(p^{-1} \nu^p+\nu)\;.\end{equation*}
Hence $\V[\nu]\in \L^{p'}(m_0)$ as soon as $\nu\in \D$ and thus, by H\"older's inequality, $\V[\nu] \rho\in \L^1(m_0)$ for every $\rho$ and $\nu$ in $\D$. 
%--------------------------
\begin{coro}[Existence of equilibria by minimisation]\label{coroexemple}
 Assume that the assumptions of {\rm Lemma~\ref{diffkanto}} hold, that $\nu\mapsto \V[\nu]$ is of the form~\eqref{eq:vou} where $f$ and $\phi$ satisfy the assumptions above, then~\eqref{varequil} admits minimisers in $\P(Y)\cap \L^p(m_0)$ so that there exists Cournot-Nash equilibria.
\end{coro}
%--------------------------
The proof is given in Appendix~\ref{acoroexemple}. Note that this in particular provides existence of equilibria results for the holiday and technological choice model examples above.  
%---------------------
\begin{rem}
Under the assumptions of the previous corollary, one can prove that the minimisers are actually bounded: indeed let $\nu$ be such a minimiser either $\nu(y)=0$ or $\nu(y)>0$ and for $m_0$-a.e. such points by the optimality condition~\eqref{eq:optcond} and~\eqref{growth1} one should have for some constant $M$
\[ a(\nu(y)^\alpha -1) \leq f(y, \nu(y))= M-\varphi^c(y)-\int_{Y} \phi(y,z) \dd\nu(z)\;.\]
Since $\varphi$ is a $c$-transform, it is continuous hence bounded on $Y$ and the integral term is bounded since $\phi$ is. We therefore have $\nu \in \L^{\infty}(m_0)$.  
\end{rem}

Let us now emphasise the role of convexity in the variational approach. As expected if $\E$ is convex, then finding equilibria and minimising $\J_\mu$ are equivalent:
%---------------------
\begin{prop}[Equivalence in the convex case]\label{equiveqmin}
Assume that the assumptions of {\rm Theorem~\ref{miniareeq}} are satisfied. If $\E$ is convex on $\D$ then the following statements are equivalent:
\begin{itemize}
\item $\nu$ solves~\eqref{varequil} and $\gamma \in \Pi_o(\mu,\nu)$,
\item $\gamma$ is an equilibrium and $\nu={\Pi_Y}_\# \gamma$.
\end{itemize}
\end{prop}
%---------------------
If moreover $\E$ is strictly convex the following uniqueness result holds:
%---------------------
\begin{coro}[Uniqueness in the strictly convex case]\label{corou}
Assume that the assumptions of {\rm Theorem~\ref{miniareeq}} are satisfied. If $\E$ is strictly convex then all equilibria share the same second marginal $\nu$. If in addition, the assumptions of {\rm Corollary~\ref{purityofCNE}} are satisfied then there is at most one Cournot-Nash equilibrium.
\end{coro}
%---------------------

As an application, let us observe that if the assumptions of Lemma~\ref{diffkanto} are satisfied and if $\V[\nu](y)=f(y, \nu(y))$  with an $f$ which is increasing in $\nu$ and satisfies~\eqref{growth1} then there exists  a unique minimiser so the previous uniqueness result holds. This applies naturally  to the technological choice equilibrium problem as well as to the holiday choice example with pure congestion or, more generally, in the case where the congestion effects dominate as explained  below. 

In the case where 
\[\E[\nu]=\int_Y F(y, \nu(y)) \dd m_0(y)+ \frac{1}{2} \iint_{Y^2} \phi(y,z) \dd\nu(y) \dd\nu(z)\] 
the second non-local term typically favours the concentration of $\nu$ (when $\phi(y,z)$ is increasing with the distance between $y$ and $z$ for instance) and it is not convex, while the congestion terms fosters dispersion and is convex. There may however be some compensation between the two terms that makes $\E$ convex. For instance, by Cauchy-Schwarz inequality, the quadratic form
\begin{equation*}
\int_{Y} \nu^2(y)\dd m_0(y) +\iint_{Y^2} \phi(y,z)\dd\nu(y)\dd\nu(z)
\end{equation*}
is positive definite hence convex as soon as 
\[\int_{Y^2} \phi^2(y,z) \dd m_0(y) \dd m_0(z)<1.\]
Whence in this case, the uniqueness result of Corollary~\ref{corou} applies.

%%%%%%%%%%%%%%%%%%%%%%%%%

\subsection{The case of Inada's condition}\label{inada}

We now  consider the case where $\V$ contains a local congestion term that satisfies an Inada-like condition:
\begin{equation}\label{eq:inada}
  \lim_{\nu \to 0^+}f(\nu)=-\infty \quad\mbox{and}\quad\lim_{\nu\to +\infty} f(\nu)=+\infty.
\end{equation}
This will imply that minimisers of~\eqref{varequil} are positive $m_0$-a.e.. The optimality condition $\varphi^c +\V[\nu]=M$, for some constant $M$, will therefore be satisfied $m_0$-a.e. which implies the regularity of $\nu$. More precisely, let us consider the case where the interaction are given for $\nu\in \P(Y)\cap \L^1(m_0)$ by the map: 
\begin{equation}
  \label{eq:noform}
  \V[\nu](y)=f(\nu(y))+ \int_{Y} \phi(y,z) \dd\nu(z)
\end{equation}
(for the sake of simplicity we have dropped the dependence in $y$ of $f$) where
\begin{itemize}
\item $\phi\in \CC(Y\times Y)$ is symmetric, 
\item  $f$ : $(0, +\infty) \mapsto \R$ is continuous increasing, locally integrable on $[0, +\infty)$ and satisfies the Inada condition~\eqref{eq:inada}.
\end{itemize}
 We then define $F$ by $F(0)=0$ and $F'=f$ so that $F$ is strictly convex, continuous on $[0, +\infty)$ and $\C^1(0,+\infty)$, bounded from below and coercive {\it i.e.} $F(\nu)/\nu\to +\infty$ as $\nu\to +\infty$. As before, for any $\nu$ in  $\P(Y)\cap \L^1(m_0)$, we define the associated cost functional 
\begin{equation*}
\E[\nu]=\left\{
  \begin{array}{ll}
   \displaystyle{ \int_{Y} F(\nu)\dd m_0+ \frac{1}{2}\iint_{Y^2}\phi(y,z)\dd\nu(y)\dd\nu(z)}&\quad\mbox{if $\displaystyle\int_{Y} F(\nu)\dd m_0<+\infty$ }\vspace{.3cm}\\
+ \infty&\quad\mbox{otherwise.}
  \end{array}
\right.
\end{equation*}

The typical example we have in mind is $f(\nu)=\log(\nu)$ and $F(\nu)=\nu\log\nu-\nu$, or simply $F(\nu)=\nu\log \nu$ since we are only dealing with probability measures. In this case, the domain of $\E$ consists of absolutely continuous measures with finite entropy. Again, we look for equilibria by solving the minimisation problem~\eqref{varequil}. The implication of Inada's condition on the interiority of minimisers is given by the following:
%-------------------------
\begin{lem}[Existence and positivity of minimisers]\label{conseqinada}
Under the above assumptions, the variational problem~\eqref{varequil} admits solutions and if $\nu$ is such a solution $\nu\geq \delta$ $m_0$-a.e. for some $\delta>0$ and $\nu\in \L^{\infty}(m_0)$.
\end{lem}
%-------------------------
The proof (see Appendix~\ref{aconseqinada}) relies on the fact that since $f(0^+)=-\infty$ the functional $\E$ abhors a vacuum. Now that we know that minimisers $\nu$ exist and are bounded from above and bounded away from $0$, under the assumptions of Lemma~\ref{diffkanto} it is easy to see, as in the previous paragraph, that they necessarily are equilibria and satisfy the optimality condition~\eqref{eq:optcond}:
\[f(\nu(y))+\int_{Y} \phi(y,z) \dd\nu(z)+\varphi^c(y) = M \]
for some constant $M$ and where as usual $\varphi$ is the Kantorovich potential between $\mu$ and $\nu$ and $\varphi^c$ is its $c$-transform.  Note that this equality is true not only $\nu$-a.e. but $m_0$-a.e., one can then invert this relation to deduce that $\nu$ coincides $m_0$ with the continuous function
\begin{equation}\label{shapeofnu}
\nu(y)=f^{-1}\left(M- \varphi^c(y) -\int_{Y} \phi(y,z) \dd\nu(z) \right).
\end{equation}
In particular there exists equilibria that have a continuous representative.\footnote{Inada's condition is actually not essential to obtain a relation of the form \eqref{shapeofnu}. Indeed, in the case of a power congestion function, $f(\nu)=\nu^{\alpha}$, $\alpha>0$, using the positive part function, one obtains a similar relation
\[\nu(y)=\left(M- \varphi^c(y) -\int_{Y} \phi(y,z) \dd\nu(z) \right)_+^{1/\alpha}.\]}
Relation~\eqref{shapeofnu} is however not very tractable in general since it involves the very indirect quantity $\varphi^c$ and an integral term. We will see in Section~\ref{hidd} how it can be simplified and reformulated as a nonlinear partial differential equation  in the case of a quadratic cost.

For the moment, the Inada condition has just enabled us to prove some further regularity properties of minimisers hence of some special equilibria. Let us summarise all this by:
%---------------------
\begin{thm}[Main results under the Inada condition]\label{summaryinada}
Let $\V$ be of the form~\eqref{eq:noform} where $f$ and $\phi$ satisfy the assumptions of this paragraph. If $\nu$ solves~\eqref{varequil} and $\gamma \in \Pi_o(\mu,\nu)$ then $\gamma$ is an equilibrium; in particular, there exists equilibria. Moreover any minimiser $\nu$ of~\eqref{varequil} is bounded and bounded away from $0$ and coincides $m_0$-a.e. with the continuous function given by~\eqref{shapeofnu}.

If, in addition, $\E$ is convex  $\gamma$ is an equilibrium if and only if $\gamma \in \Pi_o(\mu,\nu)$ where $\nu$ solves ~\eqref{varequil}.

If, in addition, $\E$ is strictly convex there is a uniqueness of the equilibrium second marginal $\nu$. 
\end{thm}
%---------------------
%%%%%%%%%%%%%%%%%%%%%%%%%%%%%%%%%%%%%%%%%%%%%%%%%%%%%%%%%%%%%
%%%%%%%%%%%%%%%%%%%%%%%%%%%%%%%%%%%%%%%%%%%%%%%%%%%%%%%%%%%%%

\section{Hidden convexity and further uniqueness results}\label{hidd}

So far, our variational approach has enabled us to prove the existence of equilibria by the minimisation problem~\eqref{varequil}. However, the previous results are not totally satisfying since in general there might exist equilibria that are not minimisers and even if we are only interested in the special equilibria obtained by minimisation,  optimality conditions like~\eqref{shapeofnu} are not tractable enough to provide a full characterisation. Under further convexity conditions that are quite stringent we have seen that equilibria necessarily are minimisers and obtained uniqueness of both. In the case where
\[\E[\nu]=\int_Y F(y, \nu(y)) \dd m_0(y)+ \frac{1}{2} \iint_{Y\times Y} \phi(y,z) \dd\nu(y) \dd\nu(z)\] 
there is a competition between the convexity of the congestion term that favours dispersion and  the non-convexity of the interaction term so that in general nothing can be said about the convexity of $\E$ in the usual sense. We shall see however, that some  convexity structure, more adapted to optimal transport, can be used to derive new uniqueness and characterisation results.  The aim of  this section is precisely to exploit some hidden convexity structure in one dimension and in higher dimensions when the cost is quadratic. This goal can be achieved thanks to the very powerful notion of displacement convexity (or some slight variant of it) due to~\cite{McC2}. In recent years, these notions of convexity, intimately linked to optimal transport, have proved to be an extremely useful and flexible tool in particular in the study of nonlinear diffusions, to our knowledge, this is the first time they are used in an economic context, see also~\cite{BMS}. We refer to Appendix~\ref{tool} for a very short presentation and Section \ref{hidddimone} for a detailed exposition in the easier one-dimensional case. Much more on this rich subject can be found in the books~\cite{AGS,villani,villani2}.

%%%%%%%%%%%%%%%%%%%%%%%%%%%%%%%%%%%%%%%%%%%%%%%%%%%%%%%%%%%%%
 \subsection{Hidden convexity in dimension one}\label{hidddimone}

 Let us start with the simple one-dimensional case where the intuition is easy to understand: the functional $\J_\mu$ is not convex with respect to $\nu$ but it is with respect to $T$, the optimal transport map from $\mu$ to $\nu$.  Let us take $X=Y=[0,1]$, $m_0$ is the Lebesgue measure on $[0,1]$, $\mu$ is absolutely continuous with respect to the Lebesgue measure, and assume that $\V[\nu]$ takes the form:
 \[\V[\nu](y)=f(\nu(y))+v(y)+ \int_{[0,1]} \phi(y,z) \dd\nu(z)\]
 and that
 \begin{itemize}
 \item the transport cost $c$ is of the form $c(x,y)=C(x-y)$ where $C$ is strictly convex and differentiable,
 \item $f$ is  increasing,
 \item $v$ is convex on $[0,1]$ and $\phi$ is convex, symmetric, differentiable and has a locally Lipschitz gradient. 
 \end{itemize}
 As already noted the corresponding cost 
 \[\E[\nu]:=\int_{0}^1 F(\nu(y)) \dd y+\frac{1}{2} \iint_{[0,1]^2} \phi(y,z) \dd \nu(y)\dd \nu(z)+ \int_{0}^1 v(y) \dd \nu(y)\]
 (with $F'=f$) is not convex in the usual sense in general and neither is the functional $\J_\mu=\W_c(\mu,.)+\E$.  
 
However, we shall see that $\J_\mu$ has  good convexity properties when one considers the following interpolation. Let $(\rho, \nu)\in \P([0,1])^2$ then there is a unique optimal transport map $T_0$ (respectively $T_1$) from $\mu$ to $\nu$ (respectively from $\mu$ to $\rho$) for the cost $c$ and it is non-decreasing (see \cite{villani}). For $t\in [0,1]$, let us define:
 \[\nu_t:={T_t}_{\#} \mu \mbox{ where }  T_t:=((1-t)T_0+t T_1)\]
 then by construction, the curve $t\mapsto \nu_t$ connects $\nu_0=\nu$ to $\nu_1=\rho$. 
 
 \begin{defi}
 A functional $\J$ : $\P(Y)\to \R\cup\{+\infty\}$ is called displacement convex whenever $t\in [0,1]\mapsto \J[\nu_t]$ is convex (for every choice of endpoints $\nu$ and $\rho$), it is called strictly displacement convex when, in addition $\J[\nu_t]<(1-t)\J[\nu]+t\J[\rho]$ when $t\in(0,1)$ and $\rho\neq \mu$. 
\end{defi}

We claim that $\J_\mu$ is strictly displacement convex; indeed, take $(\nu, \rho)$ two probability measures in the domain of $\E$ (which is convex by convexity of $F$), define $\nu_t$ as above and, let us consider the four terms in $\J_\mu$ separately:

 \begin{itemize}

 \item By definition of $\W_c$, $\nu_t$ and the strict convexity of $C$ we have  
 \[\begin{split}
 \W_c(\mu, \nu_t) &\leq \int_0^1 C(x-((1-t)T_0(x)+tT_1(x))) \dd\mu(x)\\
 &\leq (1-t) \int_0^1 C(x-T_0(x)) \dd\mu(x) +t \int_0^1 C(x-T_1(x)) \dd\mu(x)\\
 &=(1-t)\W_c(\mu, \nu)+t \W_c(\mu, \rho)
 \end{split}\]
 with a strict inequality if $t\in(0,1)$ and $\nu\neq \rho$,

 \item By construction 
 \[\int_0^1 v\dd \nu_t= \int_0^1 v(T_t(x)) \dd\mu(x) %= \int_0^1 v((1-t)T_0(x)+t T_1(x)) \dd\mu(x)
 \]
 which is convex with respect to $t$, by convexity of $v$, 

 \item Similarly
 \[\iint_{[0,1]^2} \phi(y,z) \dd \nu_t(y) \dd\nu_t(z)=\iint_{[0,1]^2} \phi(T_t(x), T_t(y)) \dd \mu(x) \dd\mu(y)\]
 is convex with respect to $t$, by convexity of  $\phi$,

 \item The convexity of the remaining congestion term is more involved.  Since $\nu_t={T_t}_\#\mu$ and $T_t$ is non-decreasing, at least formally\footnote{see~\cite{AGS,villani,villani2} for a rigorous justification.} we have $\nu_t(T_t(x))T'_t(x)=\mu(x)$, by the change of variables formula we also have
 \[\int_0^1 F(\nu_t(y)) \dd y=\int_0^1 F(\nu_t(T_t(x)))T'_t(x) \dd x= \int_0^1 F\Big(\frac{\mu(x)}{T'_t(x)}\Big)T'_t(x) \dd x\]
 and we conclude by observing that $\alpha \mapsto F(\mu(x) \alpha^{-1}) \alpha$ is convex and that $T'_t(x)$ is linear in $t$. 

 \end{itemize}

 Under the assumptions above, $\J_\mu$ is therefore strictly displacement convex and thus admits at most one minimiser (indeed if $\nu$ and $\rho$ were different minimisers, by strict displacement convexity, one would have $\J_\mu[\nu_{1/2}]<\frac12 \J_\mu[\nu]+\frac12 \J_\mu[\rho]$).   Actually more is true (see Appendix \ref{proofgeodcon} for details): if $\gamma$ is an equilibrium then its second marginal solves~\eqref{varequil} and therefore is unique. Since $c$ satisfies the generalised Spence-Mirrlees condition (see Corollary \ref{purityofCNE}), we deduce the following uniqueness result

 \begin{thm}[Uniqueness of an equilibrium by displacement convexity in dimension one] Under the assumptions above, we have the equivalence 
\begin{equation*}
  \mbox{$\nu$ is a minimiser to~\eqref{varequil} and $\gamma \in \Pi_o(\mu,\nu)$ }\quad\Leftrightarrow\quad \mbox{$\gamma$ is an equilibrium}
\end{equation*}
and since $\J_\mu$ is strictly displacement convex,  there is  uniqueness of the equilibrium (which is actually necessarily pure). 
 \end{thm}

\subsection{Hidden convexity under quadratic cost}\label{eqeqv}
The arguments of the previous paragraph can be generalised in higher dimensions when the transport cost is quadratic.
Throughout this section, we will assume the following:
\begin{itemize}
\item $X=Y=\clos{\Omega}$ where $\Omega$ is some open bounded convex subset of $\R^d$,
\item $\mu$ is absolutely continuous with respect to the Lebesgue measure (that will be the reference measure $m_0$ from now on) and has a positive density on $\Omega$,
\item  $c$ is quadratic {\it i.e.}
\[c(x,y):=\frac{1}{2} \vert x-y\vert^2, \; (x,y)\in \R^d\times \R^d,\]
\item $\V$ again takes the form
\[\V[\nu](y)=f(\nu(y))+v(y)+\int_{Y} \phi(y,z) \dd\nu(z)\]
where $v$ is convex, $f$ satisfies the assumptions of Section \ref{inada} and $\phi\in \CC(\R^{d}\times \R^d)$ is symmetric and $\C^{1,1}_\loc$ ({\it i.e.} $\C^1$ with a locally Lipschitz gradient).
\end{itemize}
Again denoting by $F$ the primitive of $f$ that vanishes at $0$, the corresponding energy reads
\[\E[\nu]=\int_Y F(\nu(y)) \dd y+\int_Y v(y) \dd \nu(y)+ \frac{1}{2} \iint_{Y^{2}} \phi(y,z) \dd \nu(z)\dd \nu(y).\]

Note that as $c$ is quadratic, Brenier's Theorem (see Theorem~\ref{brenierthm}) implies the uniqueness and the purity of optimal plans $\gamma$ between $\mu$ and an arbitrary $\nu$. The variational problem~\eqref{varequil} then takes the form
\begin{equation}\label{varequil2}
\inf_{\nu\in \P(\clos{\Omega})} \J_\mu[\nu] \quad\mbox{where}\quad \J_\mu[\nu]:=\frac{1}{2} \W_2^2(\mu, \nu)+\E[\nu]
\end{equation} 
with $\W_2^2(\mu, \nu)$ is the squared-2-Wasserstein distance between $\mu$ and $\nu$ i.e.:
\[\W_2^2(\mu, \nu):=\inf_{\gamma \in \Pi(\mu, \nu)} \int_{X^2} \vert x-y \vert^2 \dd\gamma(x,y).\]

Two more structural assumptions are needed to guarantee the strict convexity of $\J_\mu$ along generalised geodesics with base $\mu$ (see Appendix \ref{cageod} for details), namely McCann's condition:
\begin{equation}\label{geodcon1}
 \nu\mapsto \nu^d F(\nu^{-d}) \mbox{ is convex non-increasing on $(0,+\infty)$}
\end{equation}
and that $\phi$ is convex. Note that McCann's condition is satisfied by power functions with an exponent larger than $1$ as well as by the entropy $F(\nu)=\nu\log(\nu)$. 

\smallskip

In contrast with Theorem~\ref{summaryinada} where minimisers are equilibria but the reverse is not always true, convexity along generalised geodesics ensures the converse reciprocal property:
%----------------------
\begin{thm}[Equilibria and minimisers coincide, uniqueness and regularity under generalized convexity]\label{uniqgeodcon}
Under the assumptions above, we have the equivalence 
\begin{equation*}
  \mbox{$\nu$ is a minimiser to~\eqref{varequil2} and $\gamma \in \Pi_o(\mu,\nu)$ }\quad\Leftrightarrow\quad \mbox{$\gamma$ is an equilibrium}
\end{equation*}
Moreover, there exists a unique equilibrium (which is actually pure) and the second marginal $\nu$ of this equilibrium has a continuous density.
\end{thm}
%----------------------

%----------------------
%%%%%%%%%%%%%%%%%%%%%%%%%%%%%%%%
%%%%%%%%%%%%%%%%%%%%%%%%%%%%%%%%
\subsection{A partial differential equation for the equilibrium}\label{pdeeq}

In the quadratic cost framework of Paragraph \ref{eqeqv}, our aim now is to write the optimality condition~\eqref{shapeofnu} in the form of a nonlinear and non-local equation partial differential equation of Monge-Amp\`ere type. For computational simplicity, we take $v=0$ and  $f(\nu)=\log(\nu)$ but any convex,  $\C^{1,1}_{\loc}$, symmetric $v$ and any increasing $f$ satisfying  McCann's condition would lead to a similar partial differential equation.  Let us recall that the unique minimiser/equilibrium $\nu$ satisfies~\eqref{shapeofnu} and is actually characterised by this condition. In this equation, the less explicit term is $\varphi^c$. Thanks to Brenier's theorem (see Appendix \ref{quadma}), this term can be made more explicit, as follows: the Brenier map $T$ between  $\mu$ and $\nu$ is the gradient of some convex function, $T=\nabla u$ with $u$ convex, and similarly the Brenier map between $\nu$ and $\mu$ is $\nabla u^*$ where $u^*$ is  the Legendre transform of $u$. In the case of a quadratic cost, $u$ and $u^*$ are related to the Kantorovich potential  $\varphi$ and its $c$-transform $\varphi^c$ through
\begin{equation}
  \label{eq:phiu}
  \varphi(x)=\frac{1}{2} \vert x \vert^2-u(x), \quad \varphi^c(y)=\frac{1}{2} \vert y \vert^2-u^*(y),\quad \forall(x,y)\in \Omega\times \Omega\;.
\end{equation}

When $\nabla u$ is smooth enough, the Monge-Amp\`ere equation, see~\eqref{eq:mongeampere}, reads:
\begin{equation*}%\label{monge-amp1}
\mu(x)=\det(D^2u(x))\,\nu(\nabla u(x)), \quad \forall x\in \Omega
\end{equation*}
which has to be supplemented with the natural sort of boundary condition 
\begin{equation}\label{bcmongeamp}
\nabla u(\Omega)=\Omega. 
\end{equation}

We may then rewrite the optimality condition~\eqref{shapeofnu} as 
\begin{equation}\label{shapeexp}
\nu(y)=C\exp \left(- \frac{1}{2} \vert y\vert^2+u^*(y)  -\int_{Y} \phi(y,z) \dd\nu(z) \right)
\end{equation}
where $C$  is a normalisation constant that makes the total mass of the right hand side be $1$ on $\Omega$. Since $u$ is defined up to an additive constant, one may actually choose $C=1$. As $\nabla u_\# \mu=\nu$,  we first have
\[\int_{Y} \phi(\nabla u(x),z) \dd\nu(z)=\int_{\Omega} \phi(\nabla u(x), \nabla u(z)) \dd \mu (z)\;.\]
On the other hand, using the well-known convex analysis identity
\[u^*(\nabla u(x))=x \cdot \nabla u(x)-u(x)\;\]
and performing the change of variable $y=\nabla u(x)$ in~\eqref{shapeexp}, the Monge-Amp\`ere equation then becomes
\begin{multline}\label{mongeampeq}
\mu(x)=\det(D^2u(x)) \exp\left(-\frac{1}{2} \vert \nabla u(x)\vert^2 +x\cdot \nabla u(x)-u(x)\right) \times  \\
\exp\left(-\int_{\Omega} \phi(\nabla u(x), \nabla u(z))\dd \mu(z)\right) \;.
\end{multline}
The equilibrium problem is therefore equivalent to a non-local and nonlinear partial differential equation.\smallskip

This kind of partial differential equation is rather complicated. However, in dimension $1$, {\it i.e.} when $\Omega$ is an open interval, which we can assume to be $(0,1)$, the boundary condition~\eqref{bcmongeamp} is $u'(0)=0$, $u'(1)=1$ and the Monge-Amp\`ere equation~\eqref{mongeampeq} simplifies to
\[\mu(x)=u''(x) \exp\left(-\frac{1}{2}  u'(x)^2 +x\cdot u'(x)-u(x)-\int_{(0,1)} \phi(u'(x), u'(z))\dd \mu(z)\right).\]

Note that this differential equation automatically implies the strict convexity of $u$. We do not know whether this equation can be solved numerically in an iterative way but will see how the equilibrium can be computed in one dimension thanks to a convenient reformulation of~\eqref{varequil2} as explained below.

%%%%%%%%%%%%%%%%%%%%%%%%%%%%%%%%%
\subsection{Numerical computations in dimension one}\label{num1}
Let $\Omega=(0,1)$, $m_0$ be the Lebesgue measure on $(0,1)$, $\mu$ be absolutely continuous with respect to $m_0$ with a positive density still denoted $\mu$. Consider again\footnote{Again the choice of a logarithmic congestion function is not so essential, power congestion functions can be considered as well.} the variational problem
\begin{equation}\label{varpbm1d}
\inf_{\nu} \J_\mu[\nu]
\end{equation}
\begin{equation*}
  \mbox{where}\quad \J_\mu[\nu]:=\frac12 \W_2^2(\mu,\nu)+\int_0^1 \nu \log\nu + \frac{1}{2} \iint_{[0,1]^2} \phi(y,z) \dd\nu(y)\dd\nu(z)
\end{equation*}
where $\phi$ is $\C^{1,1}_{\loc}$,  convex and symmetric. Looking for $\nu$ amounts to look for its rearrangement or quantile function:
\[G(x):=\inf\{\lambda \; : \; \nu([0,\lambda])\ge x\}, \; \forall x\in (0,1).\]
Note that $G$ is non-decreasing and $G_\#m_0=\nu$. We also denote by $H$ the quantile of $\mu$. It is then well known (see \cite[Section~2.2]{villani} or~\cite[Theorem~6.0.2]{AGS}) that
\[\W_2^2(\mu,\nu)=\int_0^1 \left|G(x)-H(x)\right|^2 \dd x.\]
Moreover since  $G_\#m_0=\nu$, we have
\[ \iint_{[0,1]^2} \phi(y,z) \dd \nu(y)\dd \nu(z)= \iint_{[0,1]^2} \phi(G(x),G(\theta)) \dd x\dd \theta\;.\]
And since $\nu$ is regular the change of variable formula yields 
\[\int_0^1 \nu(x) \log(\nu(x)) \dd x =\int_0^1 \nu(G(x)) \, \log(\nu(G(x))) \,G'(x) \dd x=-\int_0^1 \log(G'(x))\dd x\,,\]
as, by definition of $G$, $\nu(G(x))\,G'(x)=1$ a.e.

Therefore reformulating~\eqref{varpbm1d} in terms of quantile consists in minimising the strictly convex functional
\[\frac12 \int_0^1 |G-H|^2  -\int_0^1 \log(G'(x))\dd x+ \frac{1}{2} \iint_{[0,1]^2} \phi(G(x),G(\theta)) \dd x\dd \theta  \]
subject to the boundary conditions $G(0)=0$, $G(1)=1$. The discretisation of this variational problem is easy to solve using standard gradient descent methods, see Figure~\ref{fig:eq}\footnote{Actually, in our simulations, we have relaxed the condition that the support of $\nu$ is $[0,1]$ {\it i.e.} the boundary conditions $G(0)=0$, $G(1)=1$.}. 

%------------
\begin{figure}[h!]
 \begin{minipage}[t]{1\linewidth}
\centering\epsfig{figure=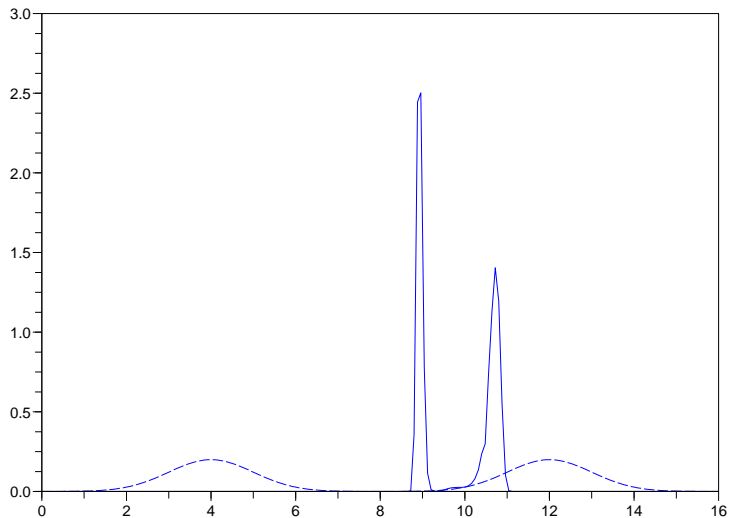,width=12cm}
 \end{minipage} 
\caption{\small The distribution $\mu$ of the agents is dash line and the solution $\nu$ to~\eqref{varpbm1d} in the case $f(x)=x^8$, $\phi(z)=10^{-4}|z|^2$ and $v=(x-10)^4$.}\label{fig:eq}
\end{figure}
%------------ 
%Note that in higher dimensions one can use a similar trick by introducing the Brenier map $\nabla u$ between $\mu$ and $\nu$: $\nabla u_\#\mu=\nu$ with $u$ convex. The functional to minimise is still convex but more complicated:
%\begin{multline*}
%  \frac12 \int_\Omega \vert \nabla u(x)-x\vert^2\dd \mu(x)-\int_{\Omega} \log\left(\frac{\det(D^2u(x))}{\mu(x)}\right)\dd\mu(x)\\
%+\frac{1}{m+1} \int_{\Omega^{m+1}} \phi(\nabla u(x_0),\ldots,\nabla u(x_m)) \dd x_0\cdots \dd x_m\;.
%\end{multline*}
%One has to take into account the convexity constraint on $u$ as well as  $\nabla u(\Omega)\subset \Omega$. The numerical approximation therefore seems to be  a much more challenging task than in dimension 1. 
%%%%%%%%%%%%%%%%%%%%%%%%%%%%%%%%%%%%%%%%%%%%%%%%%%%%%%%%%%%%%
%%%%%%%%%%%%%%%%%%%%%%%%%%%%%%%%%%%%%%%%%%%%%%%%%%%%%%%%%%%%%

\section{Concluding remarks}\label{concl}
We conclude the paper, by two remarks, for the sake of simplicity, we adopt exactly the same framework and notations, as in Section~\ref{eqeqv}. 

\subsection{Implementation by taxes}
%%%%%%%%%%%%%%%%%%%%%%%%%%%%%%%%%%%%%%%%%%%%%%%%%%%%%%%%%%%%%
By Theorem~\ref{uniqgeodcon} the unique equilibrium is the unique minimiser of the functional $\J_\mu$. It would therefore be tempting to interpret this result as a kind of welfare theorem. A simple comparison between $\J_\mu$ and the total social cost tells us however that the equilibrium is not efficient. Indeed,  the total social cost ${\rm SC}[\nu]$ is the sum of the transport cost $\W_2^2(\mu, \nu)/2$ and the additional cost $\int_Y \V[\nu](y) \,\nu(y)\dd y$ {\it i.e.} 
\[{\rm SC}[\nu]=\frac{1}{2}\W_2^2(\mu, \nu)+\int_Y f(\nu(y)) \dd\nu(y) + \int_Y v \dd\nu+ \iint_{Y^2} \phi(y,z) \dd\nu(y)\dd\nu(z)\;.\]
The second term represents the total congestion cost and the last one the total interaction cost. The functional $\J_\mu$ whose minimiser is the equilibrium has a  similar form, except that in its second term $f(\nu)\nu$ is replaced by $F(\nu)$ (with $F'=f$) and the interaction term is divided by $2$. The equilibrium corresponds indeed to the case where agents selfishly minimise their own cost 
\begin{equation*}
  c(x,.)+\V[\nu]=c(x,.)+f(\nu(.))+v + \int_{Y} \phi(y,z) \dd \nu(z)\;.
\end{equation*}
This individual minimisation has of course no reason to correctly estimate the marginal effect of individual behaviour on the total social cost. In other words, there is some gap between the equilibrium and the efficient (social-cost minimising) configurations, and, since we are dealing with a situation with externalities, this is actually not surprising. The computation of the equilibrium and the optimum can be done numerically in dimension 1 by using the same kind of numerical computations as explained in Section~\ref{num1}, see Figure~\ref{fig:eqopt}.
%------------
\begin{figure}[h!]
 \begin{minipage}[t]{.49\linewidth}
\centering\epsfig{figure=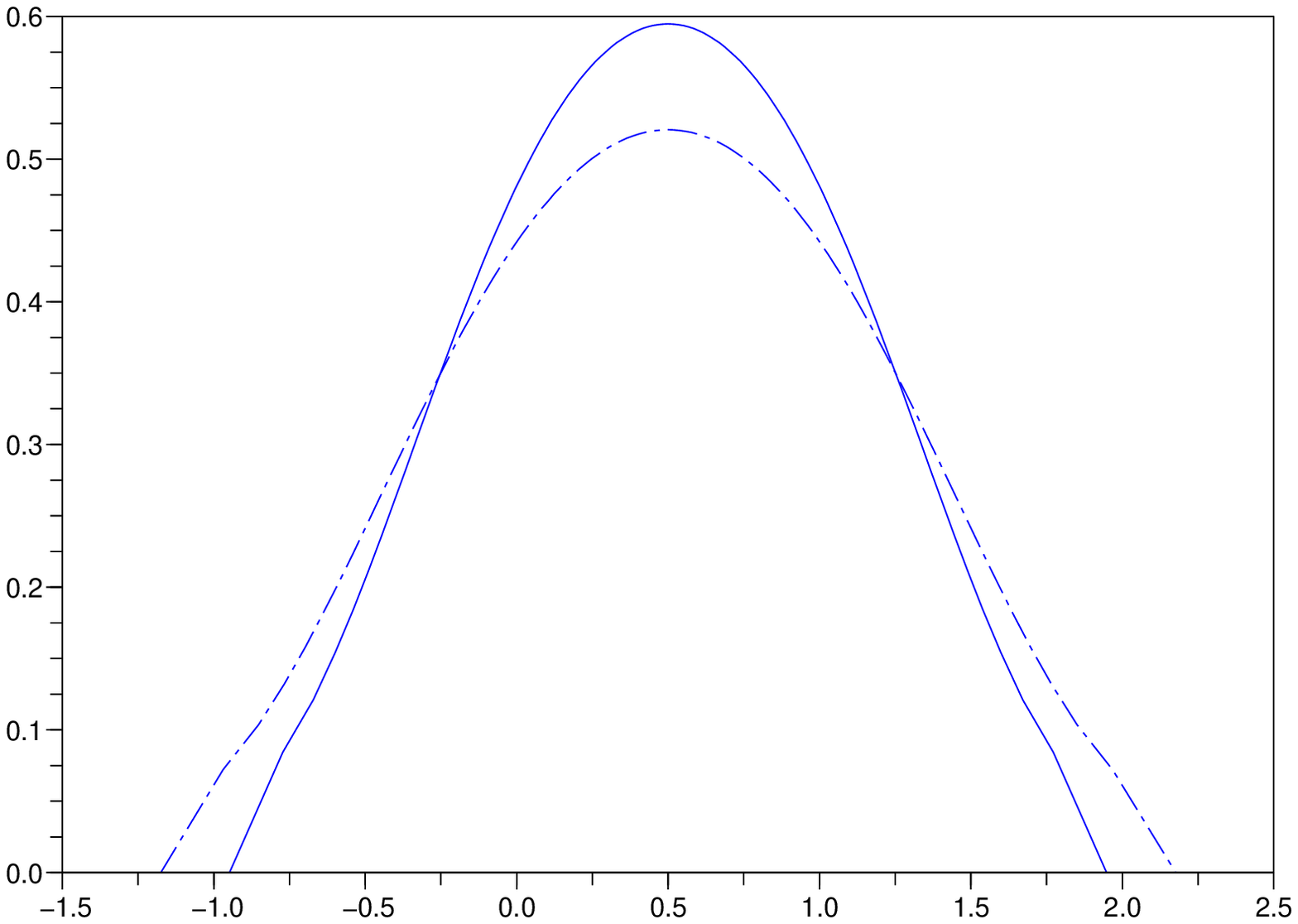,width=6.4cm}
 \end{minipage} \hfill
\begin{minipage}[t]{.49\linewidth}
\centering\epsfig{figure=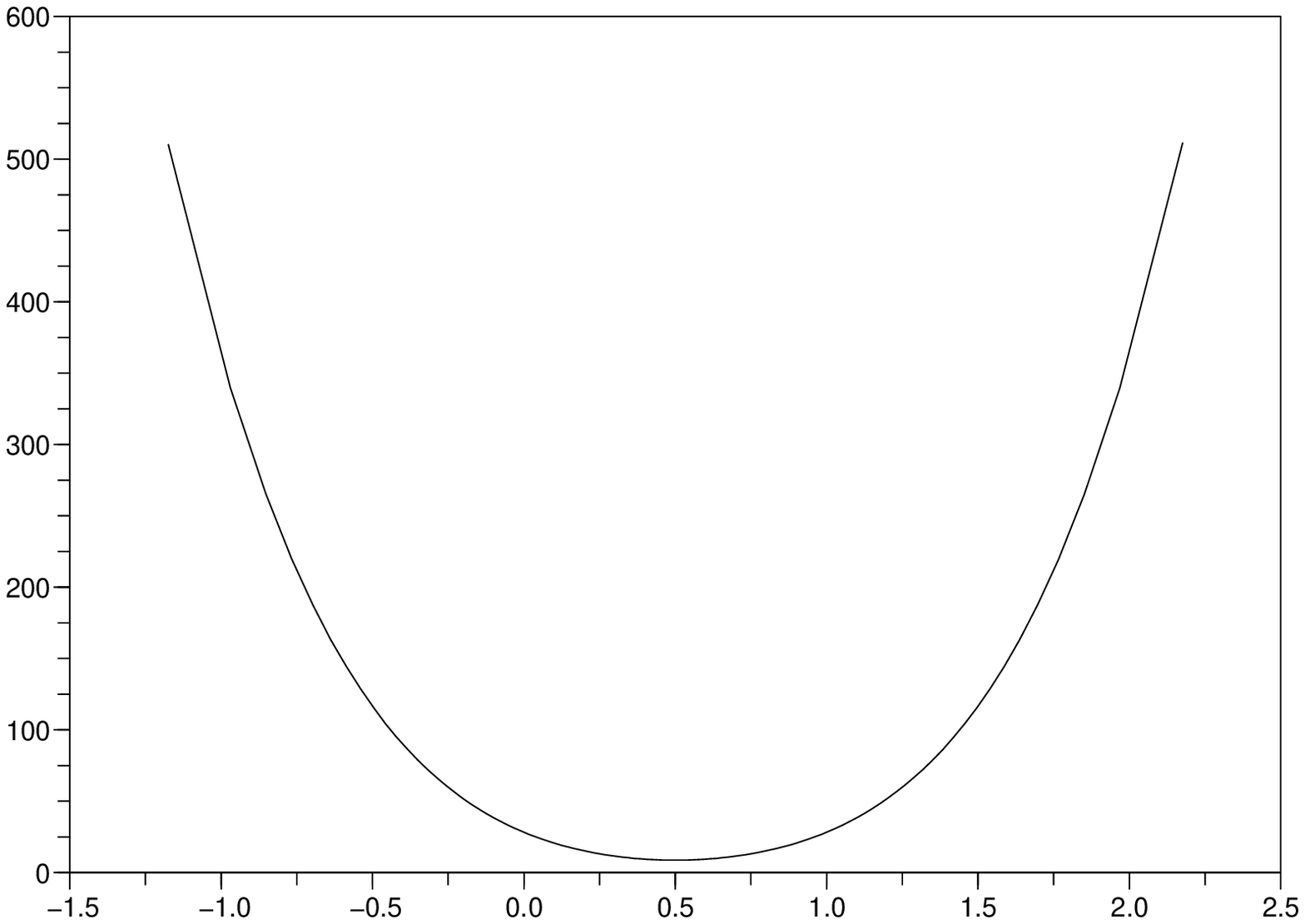,width=6.4cm}
 \end{minipage} 
\caption{\small The optimum in continuous line and the equilibrium in dash line on the left. The corresponding taxes on the right.}\label{fig:eqopt}
\end{figure}
%------------

The natural way to restore efficiency of the equilibrium is the design by some social planner of a proper system of tax/subsidies which, added to $\V[\nu]$, will implement the efficient configuration (or at least a stationary point of the social cost). Thanks to our variational approach, a tax system that restores the efficiency is easy to compute (up to an additive constant):
\[{\mathrm{Tax}}[\nu](y)= f(\nu(y))\, \nu(y)-F(\nu(y))+\int_{Y} \phi(y,z) \dd\nu(z).\]
The two terms in  ${\mathrm{Tax}}[\nu]$  represent respectively a correction to the individual estimation of congestion cost and to the individual estimation of interaction cost. A similar inefficiency of equilibria, arises in the slightly different framework of congestion games, where it is usually referred under the name \emph{cost of anarchy}, which has been extensively studied  in recent years (see~\cite{rough} and the references therein). In our Cournot-Nash context, we may similarly define the cost of anarchy as the ratio of the worst social cost of an equilibrium to the minimal social cost value:
\begin{equation*}
  \mbox{Cost of anarchy}:=\frac {\max \{ {\mathrm{SC}}[\nu^e]\; : \; \mbox{ $\nu^e$ equilibrium}\}}{\min_{\nu} {\mathrm{SC}}[\nu]}\;.
\end{equation*}
In the previous numerical example of Figure~\ref{fig:eqopt}, both the equilibrium and the optimum are unique and the cost of anarchy can be numerically computed as being approximately $1.8$.
%%%%%%%%%%%%%%%%%%%%%%%%%%%%%%%%%%%%%%%%%%%%%%%%%%%%%%%%%%%%%
\subsection{A dynamical perspective}
Instead of minimising $\J_\mu$ directly, we may think that agents start with some distribution of strategies (that is not an equilibrium) and adjust it with time by a sort of gradient descent dynamics to decrease their individual cost dynamically. At least formally, a way to reach the equilibrium (or minimiser of $\J_\mu$) is then to put it into the dynamical perspective of the minimising movement scheme as follows. Fix a time step $\tau>0$ and start with an initial configuration of strategies $\nu_0$. The first step of the minimising movement scheme selects a new distribution of strategies $\nu_1$ close to $\nu_0$ (in $\W_2$) but also decreasing $\J_\mu$ by
\[\nu_1 \in \argmin_\nu \left\{\frac{1}{2\tau}\W^2_2(\nu_0, \nu)+\J_\mu[\nu]\right\}\;.\]
And then it iterates the process by choosing 
\begin{equation}
  \label{eq:jkoc}
  \nu_{k+1} \in \argmin_\nu \left\{\frac{1}{2\tau}\W^2_2(\nu_k, \nu)+\J_\mu[\nu]\right\}\;.
\end{equation}
This sort of Wasserstein Euler scheme  was first introduced in~\cite{jko} for the Fokker-Planck equation. Under suitable conditions it is possible to pass the continuous limit $\tau \to 0^+$ in the minimising scheme~\eqref{eq:jkoc} and prove that the solution converges in some sense to the continuous evolution equation
\begin{equation*}
\left\{
  \begin{array}{l}
     \partial_t \nu+ \dive\left(-\nu \nabla\left( \dfrac{ \delta \J_\mu}{\delta \nu}\right)\right)=0, \vspace{.23cm}\\
\nu_{t=0}=\nu_0
  \end{array}
 \right.
\end{equation*}
which is the gradient flow of $\J_\mu$ in the Wasserstein space (see \cite{AGS} for a detailed exposition of the theory). By construction $\J_\mu$ is a Lyapunov function of this equation and even though the equation may have non-unique solutions, by Lyapunov theory, it can be shown under appropriate conditions that its trajectories converge in large time to the unique minimiser of $\J_\mu$ {\it i.e.} the equilibrium. If we go back to the individual level, it can be shown that the equation above corresponds to the fact that each agent modifies her strategy according to the gradient flow of her individual cost.%\new{This dynamic procedure also raises the question of the stability, the convergence and the speed of convergence toward the equilibria}. 

We obtain a sequence of densities which converges to the equilibrium, see Figure~\ref{fig:var5}. The descent algorithm is very fast and the computed equilibrium  is very stable with respect to the initial density for the gradient descent as shown in the left hand figure of Figure~\ref{fig:var6}.

%------------
\begin{figure}[ht!]
 \begin{minipage}[t]{.49\linewidth}
\centering\epsfig{figure=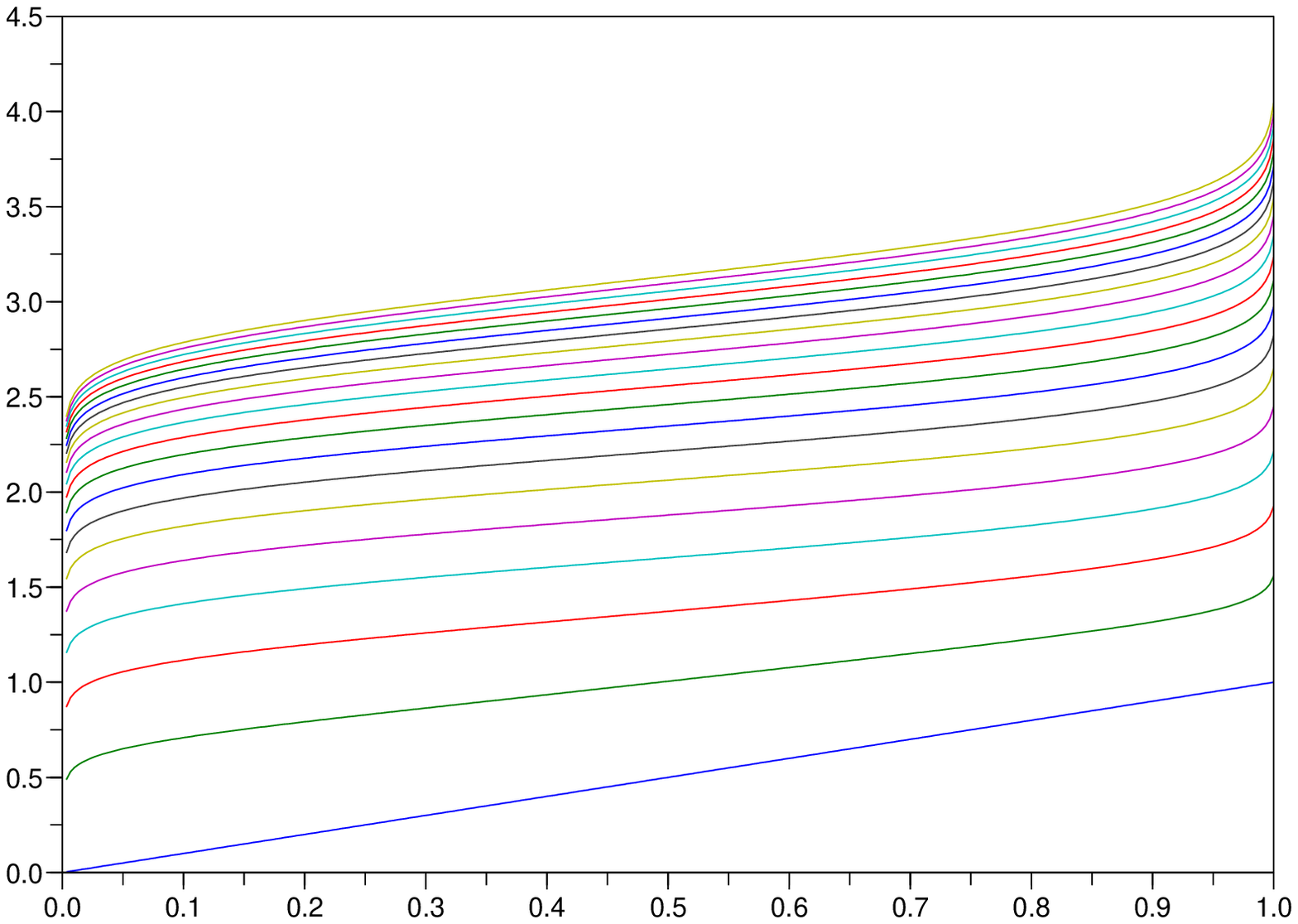,width=6.4cm}
 \end{minipage}\hfill
 \begin{minipage}[t]{.49\linewidth}
\centering\epsfig{figure=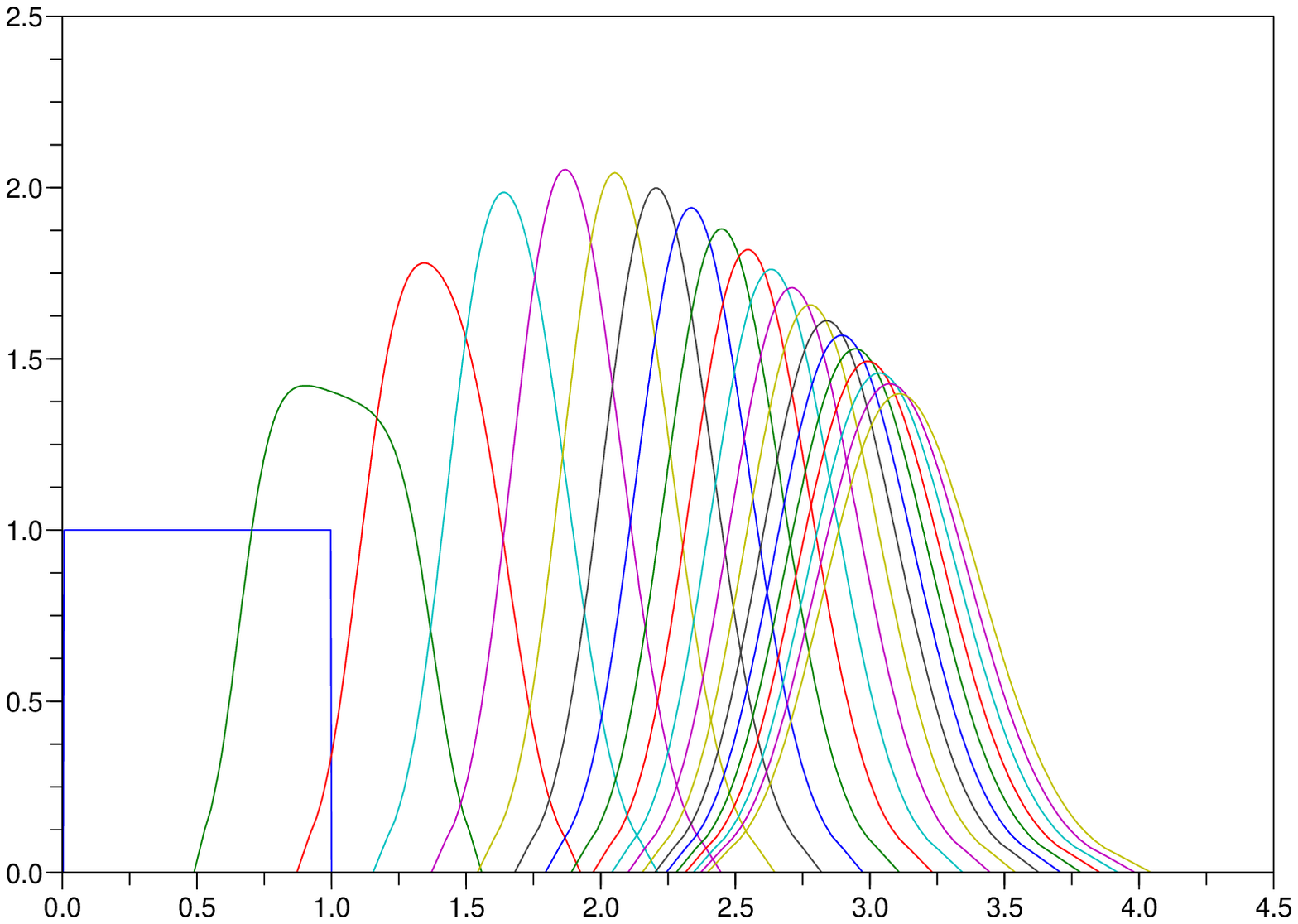,width=6.4cm}
 \end{minipage}
   \caption{\small Convergence and stabilisation toward the equilibrium in the case of a logarithmic congestion, cubic interaction, and a potential $v(x):=(x-5)^3$ with $\un_{[0,1]}$ as initial guess. The inverse of the cumulative function on the left and the corresponding density on the right.}\label{fig:var5}
\end{figure}
%------------
%------------
\begin{figure}[ht!]
 \begin{minipage}[t]{.49\linewidth}
\centering\epsfig{figure=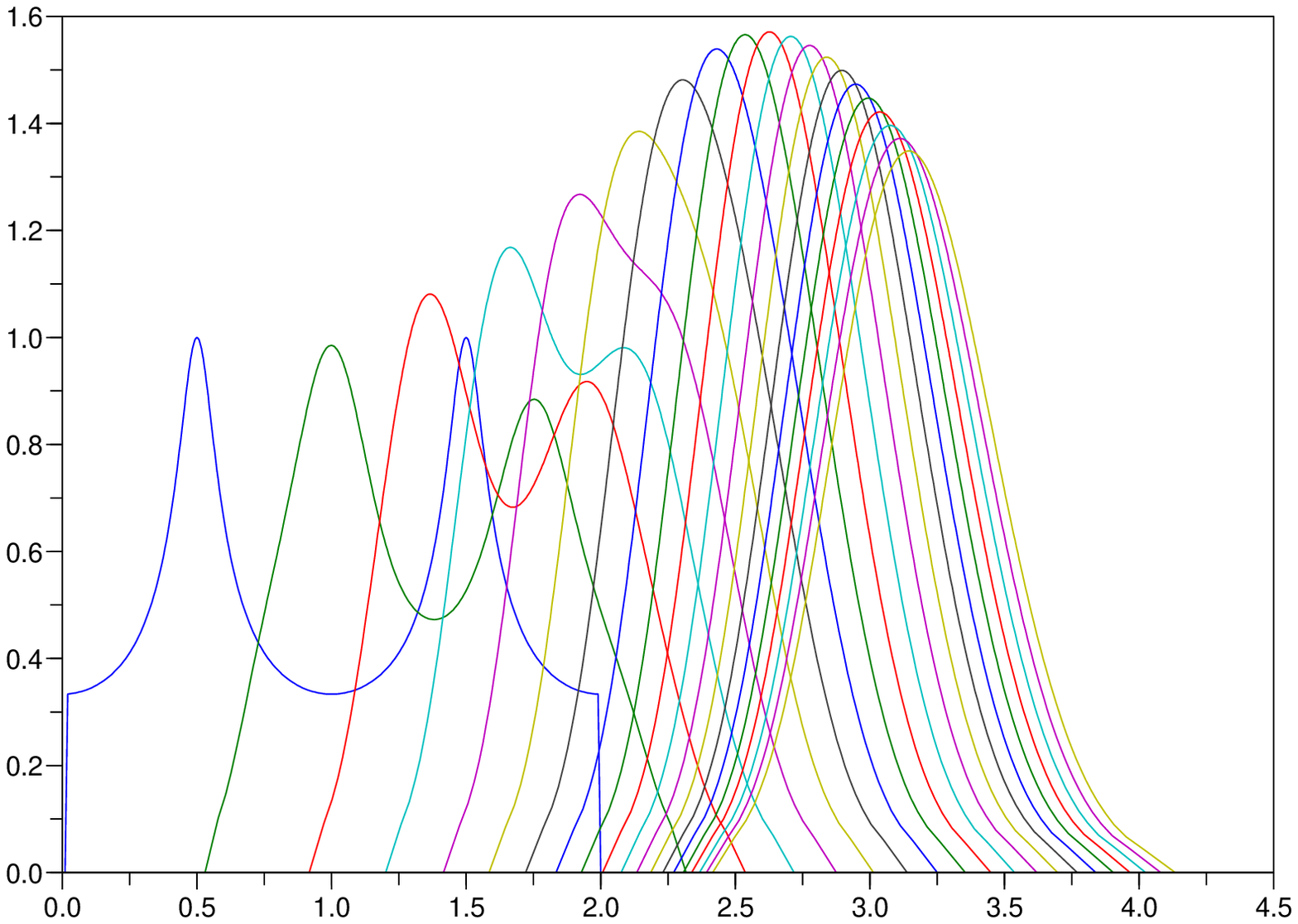,width=6.4cm}
 \end{minipage}\hfill
 \begin{minipage}[t]{.49\linewidth}
\centering\epsfig{figure=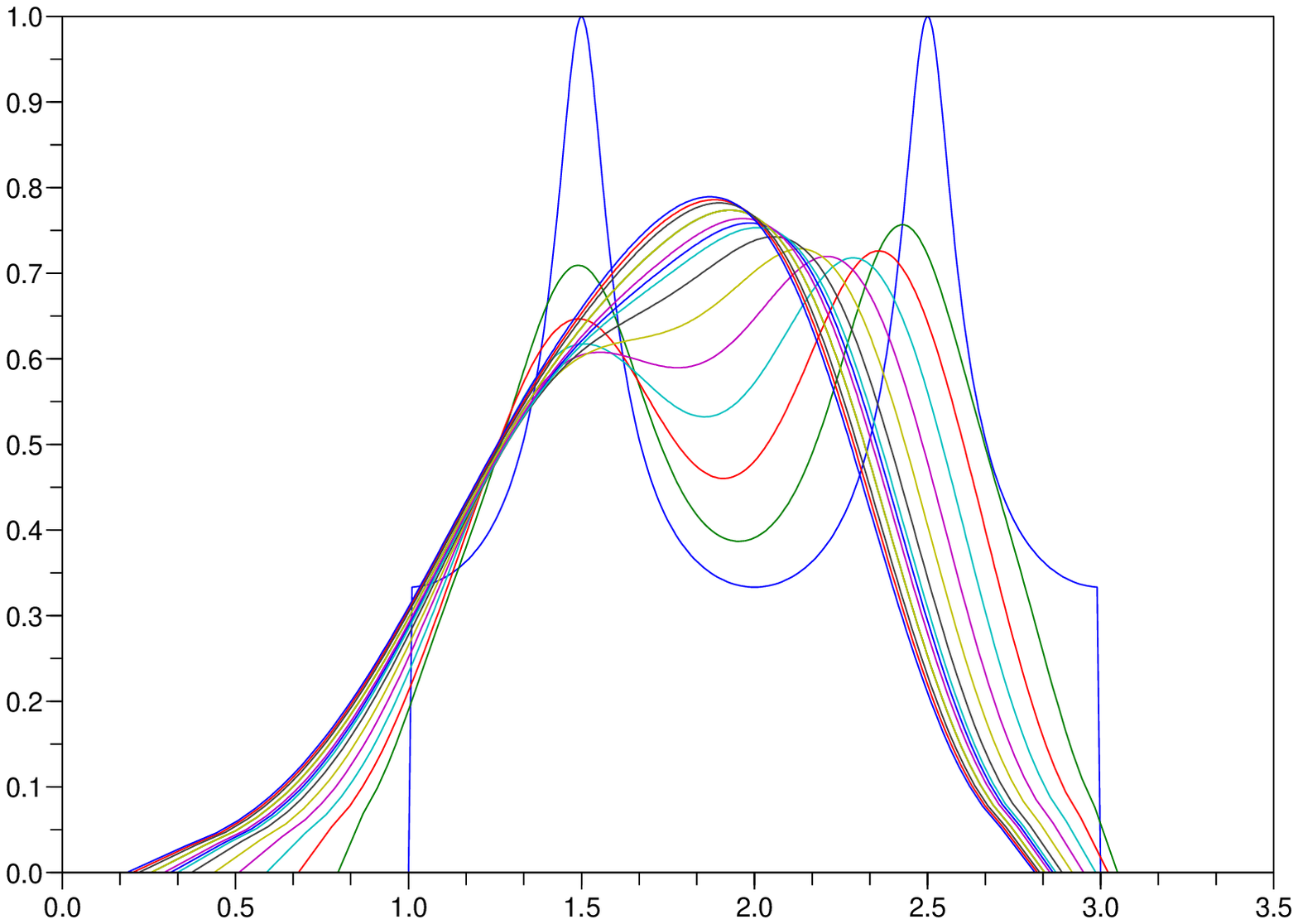,width=6.4cm}
 \end{minipage}
   \caption{\small Convergence and stabilisation toward the equilibrium in the case of a logarithmic congestion, cubic interaction with an initial guess made of two bumps. The potential is $v(x):=(x-5)^3$ on the left and $v(x):=(x-1/2)^3$ on the right.}\label{fig:var6}
\end{figure}
%------------
%%%%%%%%%%%%%%%%%%%%%%%%%%%%%%%%%%%%%%%%%%%%%%%%%%%%%%%%%%%%%
%%%%%%%%%%%%%%%%%%%%%%%%%%%%%%%%%%%%%%%%%%%%%%%%%%%%%%%%%%%%%
%%%%%%%%%%%%%%%%%%%%%%%%%%%%%%%%%%%%%%%%%%%%%%%%%%%%%%%%%%%%%
%%%%%%%%%%%%%%%%%%%%%%%%%%%%%%%%%%%%%%%%%%%%%%%%%%%%%%%%%%%%%
\bigskip

\noindent{\bf{Acknowledgements}}  The authors acknowledge the support of the Agence Nationale de la Recherche through the projects ANR-09-JCJC-0096-01 EVaMEF and ANR-07-BLAN-0235 OTARIE. The authors wish to thank Jocelyn Donze, Andr\'e Grimaud, Michel Le Breton, J\'er\^ome Renault, Fran\c{c}ois Salani\'e and the participants to the IAST LERNA - Eco/Biology and Chicago University Seminars for many interesting and fruitful discussions about the present work. 
%%%%%%%%%%%%%%%%%%%%%%%%%%%%%%%%%%%%%%%%%%%%%%%%%%%%%%%%%%%%%
%%%%%%%%%%%%%%%%%%%%%%%%%%%%%%%%%%%%%%%%%%%%%%%%%%%%%%%%%%%%%
\appendix

\section{The optimal transport toolbox}\label{tool}
This appendix just gives some basic results from optimal transport theory that we have used in the paper, for a detailed exposition of this rich and rapidly developing subject, we refer the interested reader to the very accessible textbook~\cite{villani} or~\cite{AGS,villani2} or, the more probability-oriented textbook ~\cite{raru}.
%%%%%%%%%%%%%%%%%%%%%%%%%%%%%%%%%%%%%%%%%%%%%%%%%%%%%%%%%%%%%

\subsection*{Kantorovich duality}

Let $X$ and $Y$ be two compact spaces equipped respectively with the Borel probability measures $\mu\in \P(X)$ and $\nu\in \P(Y)$. For $\mu\in \P(X)$ and $T$, Borel: $X\to Y$, $T_\#\mu$ denotes the {\it push forward} (or image measure) of $\mu$ through $T$ which is defined by $T_\#\mu(B)=\mu(T^{-1}(B))$ for every Borel subset $B$ of $Y$ or equivalently by the change of variables formula
\begin{equation}\label{changevar}
\int_Y \varphi \dd T_\#\mu=\int_X \varphi(T(x)) \dd\mu(x), \; \forall \varphi \in \C(X). 
\end{equation}
A transport map between $\mu$ and $\nu$ is a Borel map such that $T_\#\mu=\nu$. Now, let $c\in \C(X\times Y)$ be some transport cost function, the {\it Monge optimal transport} problem for the cost $c$  consists in finding a transport $T$ between $\mu$ and $\nu$ that minimises the total transport cost $\int_X c(x, T(x)) \dd\mu(x)$. A minimiser is then called an {\it optimal transport}. Monge problem is in general difficult to solve (it may even be the case that there is no transport map, for instance it is impossible to transport one Dirac mass to a sum of distinct Dirac masses), this is why  Kantorovich relaxed Monge's formulation as 
\begin{equation}\label{relaxkanto}
\W_c(\mu, \nu):=\inf_{\gamma \in \Pi(\mu, \nu)} \int_{X\times Y} c(x,y) \dd\gamma(x,y) 
\end{equation}
where $\Pi(\mu, \nu)$ is the set of transport plans between $\mu$ and $\nu$ {\it i.e.}  Borel probability measures on $X\times Y$ having $\mu$ and $\nu$ as marginals. Since $\Pi(\mu, \nu)$ is weakly $*$ compact and $c$ is continuous, it is easy to see that the infimum of the linear program defining $\W_c(\mu, \nu)$ is attained at some $\gamma$, such optimal $\gamma$'s are called {\it optimal transport plans} (for the cost $c$) between $\mu$ and $\nu$. If there is an optimal $\gamma$ which is induced by a {\it transport map} {\it i.e.} is of the form $\gamma=(\id, T)_\#\mu$ for some transport map $T$ then $T$ is obviously an optimal solution to Monge's problem. Another advantage of the linear relaxation is that it possesses a dual formulation that can be very useful. This dual formulation consists in maximising the linear form $\int_X \varphi \dd\mu+\int_Y \psi \dd\nu$ among all pairs $(\varphi, \psi)\in \C(X)\times \C(Y)$ such that $\varphi(x)+\psi(y)\leq c(x,y)$, it is easy to see that this can be reformulated as a maximisation over $\varphi$ only:
\begin{equation}\label{dualkanto}
\W_c(\mu, \nu):=\sup_{\varphi \in \C(X)} \Big\{ \int_X \varphi \dd\mu + \int_Y \varphi^c \dd\nu \Big\}
\end{equation}
where $\varphi^c$ is the $c$-concave transform of $\varphi$ {\it i.e.} 
\[\varphi^c(y):=\min_{x\in X} \{c(x,y)-\varphi(x)\}, \; \forall y\in Y.\]
Formula~\eqref{dualkanto} is usually called {\it Kantorovich duality formula} and a maximiser $\varphi$ in~\eqref{dualkanto} is called a {\it Kantorovich potential} between $\mu$ and $\nu$ for the cost $c$. The existence of Kantorovich potentials under our assumptions is well-known (see~\cite{villani,villani2,raru}) and we observe that if $\varphi$ is a Kantorovich potential then so is $\varphi+C$ for every constant $C$. 

We have used in Section \ref{varapp} the following result on the uniqueness of the Kantorovich potential and the differentiability of $\W_c(\mu, \nu)$ with respect to $\nu$:

\begin{lem}\label{diffkanto}
Assume that $X=\clos{\Omega}$ where $\Omega$ is some open bounded connected subset of $\R^d$ with negligible boundary, that $\mu$ is equivalent to the Lebesgue measure on $X$ (that is both measures have the same negligible sets) and that for every $y\in Y$, $c(.,y)$ is differentiable with $\nabla_x c$ bounded on $X\times Y$, let $\nu\in \P(Y)$ then there exists a unique (up to an additive constant) Kantorovich potential $\varphi$ between $\mu$ and $\nu$ and for every $\rho\in \P(Y)$ one has 
\[\lim_{\eps\to 0^+} \frac{\W_c(\mu, \nu+\eps(\rho-\nu))-\W_c(\mu, \nu)}{\eps}=\int_Y \varphi^c \dd (\rho-\nu).\]
\end{lem}

\begin{proof}
The proof of the uniqueness of the  Kantorovich potential  $\varphi$ between $\mu$ and $\nu$ up to an additive constant can be found for instance in~\cite[Proposition 6.1]{gcie}. As a normalisation we choose the potential $\varphi$ such that $\varphi(x_0)=0$ where $x_0$ is some given point of $X$. To shorten notations, set $\nu_\eps=\nu+\eps(\rho-\nu)$, thanks to Kantorovich duality formula~\eqref{dualkanto} we have
\begin{equation}\label{liminfwass}
\eps^{-1}[\W_c(\mu, \nu_\eps)-\W_c(\mu, \nu)] \geq \int_Y  \varphi^c \dd (\rho-\nu)
\end{equation}
and similarly if $\varphi_\eps$ denotes the Kantorovich potential between $\mu$ and $\nu_\eps$ such that $\varphi_\eps(x_0)=0$, we have
\[\eps^{-1}[\W_c(\mu, \nu_\eps)-\W_c(\mu, \nu)] \leq \int_Y  \varphi_\eps^c \dd (\rho-\nu).\]
Now it is well-known  that $(\varphi_\eps)_\eps$ is bounded and uniformly equi-continuous uniformly with respect to $\eps$ hence, thanks to Ascoli's Theorem, up to a sub-sequence, it converges uniformly to some $\clos{\varphi}\in \C(X)$ such that $\clos{\varphi}(x_0)=0$ and it is easy to see that $\clos{\varphi}$ is a Kantorovich potential between $\mu$ and $\nu$ so that $\varphi=\clos{\varphi}$ and $\varphi_\eps^c$ converges to $\varphi^c$. We then have 
\[\limsup_{\eps\to 0^+}\eps^{-1}[\W_c(\mu, \nu_\eps)-\W_c(\mu, \nu)] \leq \int_Y  \varphi^c \dd (\rho-\nu).\]
The desired result thus follows from~\eqref{liminfwass}. 
\end{proof}

When $X=Y$ and denoting by $d$ the distance on $Y$, for $p\in [1, +\infty[$, the $p$-Wasserstein distance between $\mu\in \P(X)$ and $\nu\in \P(X)$ is by definition
\begin{equation}
\W_p(\mu, \nu):= \Big( \inf_{\gamma\in \Pi(\mu, \nu)} \Big\{ \int_{X\times Y} d(x,y)^p \dd\gamma(x,y)\Big\}   \Big)^{1/p}
\end{equation}
The Wasserstein distances are indeed distances and they metrise the weak $*$ topology of $\P(Y)$. 
For $p=1$, it is well-known that  the Kantorovich duality formula can be rewritten as
\[\W_1(\mu, \nu)=\sup \Big\{ \int_X \varphi \dd(\mu-\nu) \; : \; \varphi \mbox{ 1-Lipschitz} \Big\} \]
so that for every Lipschitz continuous function $\varphi$ on $X$, one has 
\[\Big\vert  \int_X \varphi \dd(\mu-\nu)   \Big \vert \leq \lip(\varphi, X) \W_1(\mu, \nu),\]
an inequality we will use several times later on.  As a simple illustration of the interest of the distance $\W_1$, let us equip $Y^m$ with the distance $(x,y)=(x_1, \ldots, x_m, y_1,\ldots y_m)\mapsto d_m(x,y):=\sum_{k=1}^m d(x_k, y_k)$ for $\nu$ and $\theta$ in $\P(Y)$, let $\gamma$ be an optimal transport plan between $\mu$ and $\nu$ for $\W_1$ then since $\gamma^{\otimes m}:=\gamma\otimes \cdots \otimes \gamma$ has marginals $\nu^{\otimes m}$ and $\theta^{\otimes m}$, we have
\[\W_1(\nu^{\otimes m}, \theta^{\otimes m})\leq \int_{Y^m\times Y^m} d_m \dd\gamma^{\otimes m}=m\W_1(\nu, \theta)\;.\]
Which shows in particular that if $(\nu_n)_n$ weakly $*$ converges to $\nu$ then  $(\nu_n^{\otimes m})_n$ weakly $*$ converges to $\nu^{\otimes m}$ {\it i.e.} $\int_{Y^m} \varphi \dd\nu_n^{\otimes m}$ converges to $\int_{Y^m} \varphi \dd\nu^{\otimes m}$ as $n\to \infty$ for every $\phi\in \C(Y^m)$.

Of particular interest is also the quadratic case $p=2$ in an euclidean setting for which a brief summary of the main results used in the paper is given in the next paragraphs.

\subsection*{The quadratic case and Monge-Amp\`ere equation}\label{quadma}

We now restrict ourselves to the quadratic case, the solution of the quadratic optimal transport problem is due to Yann Brenier whose path-breaking paper \cite{bre} totally renewed the field of optimal transport and was the starting point of an extremely active stream of research since the 90's.

\begin{thm}[Brenier's theorem]\label{brenierthm}
Let $\mu\in \P(\R^d)$ be absolutely continuous with respect to the Lebesgue measure  and compactly supported and $\nu\in \P(\R^d)$ be compactly supported, then the quadratic optimal transport problem
\[W_2^2(\mu, \nu):=\inf_{\gamma\in \Pi(\mu, \nu)} \iint_{\R^d\times \R^d} \vert x-y\vert^2 \dd\gamma(x,y)\] 
possesses a unique solution $\gamma$ which is in fact a Monge solution $\gamma=(\id, T)_\#\mu$. Moreover $T=\nabla u$ $\mu$-a.e. for some convex function $u$ and $\nabla u$ is the unique (up to $\mu$-a.e. equivalence)  gradient of a convex function transporting $\mu$ to $\nu$; $T=\nabla u$ is called the {\it Brenier map} between $\mu$ and $\nu$. 
\end{thm}

In fact the previous theorem holds under much more general assumptions (it is enough that $\mu$ and $\nu$ have finite second moments and that $\mu$ does not charge sets of Hausdorff dimension less than $d-1$, see \cite{McC3} or \cite{villani}). Brenier's theorem roughly says that there is a unique optimal transport for the quadratic cost and that it is characterised by the fact that it is of the form $\nabla u$ with $u$ convex, in other words, solving $\nabla u_\# \mu=\nu$  with $u$ convex determines $\nabla u$ uniquely $\mu$-a.e.. When we have additional regularity, {\it i.e.} when $\mu$ and $\nu$ have regular densities (still denoted $\mu$ and $\nu$) and $\nabla u$ is a diffeomorphism between the support of $\mu$ and that of $\nu$, thanks to the change of variables formula, we find that $u$ solves the Monge-Amp\`ere partial differential equation:
\begin{equation}\label{eq:mongeampere}
\mu=\nu(\nabla u) \det(D^2 u).
\end{equation}
A deep regularity theory due to Luis Caffarelli: \cite{caff1,caff2} implies that the Brenier map is a smooth diffeomorphism when in addition $\mu$ and $\nu$ are smooth, bounded away from $0$ and have  convex supports, in particular the Monge-Amp\`ere equation is satisfied in this case which justifies the computations of Section~\ref{pdeeq}.

\subsection*{Convexity along generalised geodesics}\label{cageod}

The last ingredient from optimal transport theory that we have used (in Section \ref{hidd}) is the powerful notion of displacement  convexity along generalised geodesics due to Ambrosio, Gigli-Savar\'e \cite{AGS}\footnote{Actually this notion of convexity is a slight variant of the notion of displacement convexity which first appeared in the seminal work of  \cite{McC2}. It is known that $\W^2_2(\mu, \nu)$, as a function of $\nu$ is not displacement convex in the sense of McCann (see example 9.1.5. in \cite{AGS}) and this is the very reason why, following Ambrosio, Gigli and Savar\'e, we consider convexity along generalised geodesics with base $\mu$ rather than the initial notion of McCann. Let us however indicate that, in dimension one, both notions coincide.}. As in Section \ref{eqeqv}, we assume that $X=Y=\clos{\Omega}$ where $\Omega$ is some open bounded convex subset of $\R^d$,  that the cost is quadratic, that $m_0$ is the Lebesgue measure on $X$  and that $\mu$ is absolutely continuous with respect to $m_0$ and has a positive density on $\Omega$. In particular for every $\nu\in \P(X)$, the Brenier's map between $\mu$ and $\nu$ is well-defined. Generalised geodesics with base $\mu$ for the Wasserstein distance $\W_2$ and the corresponding notion of convexity are defined as follows

\begin{defi}[Convexity along generalised geodesics]\label{defcgg}
Let $\nu\in \P(X)$, $\rho\in \P(X)$, let $T_0$ be the Brenier's map between $\mu$ and $\nu$ and let $T_1$  be Brenier's map between $\mu$ and $\rho$, the generalised geodesic with base $\mu$  between $\nu$ and $\rho$ is the curve of measures $t\in [0,1]\mapsto \nu_t:=((1-t)T_0+tT_1)_\# \mu$. The functional $\J$:  $\P(X) \to \R\cup \{+\infty\}$ is called \emph{convex along generalised geodesics with base $\mu$} if for every pair of  endpoints $\nu$ and $\rho$ in $\P(X)$ and for every $t\in[0,1]$, one has 
\[\J[\nu_t]\leq (1-t) \J[\nu]+t\J[\rho].\]
If, in addition, the previous inequality is strict for $t\in(0,1)$ and $\rho\neq \nu$, $\J$ is called \emph{strictly convex along generalised geodesics with base $\mu$}.
\end{defi}

In Section \ref{eqeqv}, we were interested in the strict convexity along generalised geodesics with base $\mu$ of the functional $\J_\mu$ defined by~\eqref{varequil2}. As in Paragraph \ref{hidddimone}, the convexity of
\[t\mapsto \int_{Y\times Y} \phi(y,z) \dd\nu_t (y) \dd\nu_t(z) \qtext{ and } t\mapsto \int_Y v \dd\nu_t\]
directly follows from the convexity of $\phi$ and $v$ respectively.  As for the convexity of $t\mapsto \int_Y F(\nu_t(y)) $ under  McCann's condition: 
\begin{equation}\label{geodcon2}
 \nu\mapsto \nu^d F(\nu^{-d}) \mbox{ is convex non-increasing on $(0,+\infty)$},
\end{equation}
it follows from~\cite[Proposition 9.3.9]{AGS}. Finally, in the functional $\J_\mu$  defined by \eqref{varequil2}, we had the term $\W^2_2(\mu, .)$, to see that it is strictly convex along generalised geodesics with base $\mu$, we can proceed exactly as we did for $\W_c(\mu,.)$ in dimension one in Paragraph~\ref{hidddimone}.

\section{Proofs of the results}\label{proof}

\subsection{Proof of Theorem~\ref{miniareeq}}\label{aminiareeq}
Let $\nu$ be a solution of~\eqref{varequil}, $\rho \in \D$ and $\eps\in (0,1)$,  we then have $\eps^{-1}(\J_\mu[\nu + \eps(\rho-\nu)]-\J_\mu[\nu])\geq 0$. Using the fact that $\V[\nu]$ is the first variation of $\E$ and Lemma~\ref{diffkanto}, we thus get
\begin{equation}\label{sinumero}
  \int_{Y} (\varphi^c +\V[\nu])\dd\rho\ge \int_{Y} (\varphi^c +\V[\nu])\dd\nu
\end{equation}
where $\varphi$ is a Kantorovich potential between $\mu$ and $\nu$ (it is unique up to an additive constant by Lemma~\ref{diffkanto} and this constant plays no role since $\rho$ and $\nu$ have the same mass). Minimising the left-hand side of~\eqref{sinumero}, with respect to $\L^p$ probabilities $\rho$, this yields that $\nu$-a.e.
\[\varphi^c +\V[\nu]=\inf_{\rho\in \D} \int_{Y}  (\varphi^c +\V[\nu])\dd\rho=M\]
where $M:=\mathrm{Essinf}  (\varphi^c +\V[\nu])$ denotes the essential infimum of $(\varphi^c +\V[\nu])$ {\it i.e.} the largest constant that bounds $(\varphi^c +\V[\nu])$ from below $m_0$-a.e.. Since $\gamma$ is an optimal transport plan we have $c(x,y)=\varphi(x)+\varphi^c(y)$ $\gamma$-a.e. whereas by definition $c(x,z)\geq \varphi(x)+\varphi^c(z)$ for all $(x,z)\in X\times Y$.  We thus have
\begin{equation*}
\left\{
  \begin{array}{ll}
   c(x,z)+\V[\nu](z)\geq M+\varphi(x) \quad&\mbox{for all $x\in X$ and $m_0$-a.e. $z\in Y$}\vspace{.1cm}\\
c(x,y)+\V[\nu](y)=M+\varphi(x)\quad&\mbox{for $\gamma$-a.e. $(x,y)$.}
  \end{array}
\right.
\end{equation*}
This proves that $\gamma$ is a Cournot-Nash equilibrium.  

%%%%%%%%%%%%%%%%%%
%%%%%%%%%%%%%%%%%%%%%%%%%%%%%%%%%
\subsection{Proof of Corollary~\ref{coroexemple}}\label{acoroexemple}

Thanks to Theorem \ref{miniareeq}, it is enough to prove that~\eqref{varequil} admits solutions and to recall that the set $\Pi_o(\mu, \nu)$ is nonempty. Let $(\nu_n)_n$ be a minimising sequence of~\eqref{varequil}. Thanks to the growth condition~\eqref{growth1}, $(\nu_n)_n$ is bounded in $\L^p(m_0)$. It thus admits a (not relabelled) sub-sequence that converges weakly in $\L^p(m_0)$ (and thus in particular weakly $*$ in $\P(Y)$) to some $\nu\in \L^p(m_0)$. By the convexity of $F(y,.)$, the first term in $\E$ is lower-semi continuous for the weak topology of $\L^p(m_0)$.  By the continuity of $\phi$ the second term in $\E$ is continuous for the weak-$*$ topology of $\P(Y)$. Finally, the lower-semi continuity of $\W_c(\mu,.)$ for the weak-$*$ topology straightforwardly follows from the Kantorovich duality formula~\eqref{dualkanto}. We thus have
\begin{equation*}
  \inf_{\nu}\J_\mu(\nu)=\liminf_n \left\{\W_c(\mu,\nu_n)+\E[\nu_n]\right\}\ge \W_c(\mu,\nu)+\E[\nu]\;.
\end{equation*}
So that $\nu$ solves~\eqref{varequil}.
%%%%%%%%%%%%%%%%%%%%%%%%%%%%%%%%%
\subsection{Proof of Proposition~\ref{equiveqmin}}\label{aequiveqmin}
Assume that $\gamma$ is an equilibrium and let $\nu$ be its second marginal. Let then $\varphi$ be a Kantorovich potential between $\mu$ and $\nu$ such that
 \begin{equation}\label{eqmk2}
\left\{
   \begin{array}{ll}
     \varphi^c +\V[\nu] \geq 0  \quad&\mbox{ $m_0$-a.e.} \vspace{.1cm}\\
\varphi^c +\V[\nu] = 0 \quad&\mbox{ $\nu$-a.e.}
   \end{array}
\right.
 \end{equation}
 Let $\rho\in \D$, thanks to the Kantorovich duality formula~\eqref{dualkanto}, we first have
\begin{align*}%\label{diffwc}
\W_c(\mu, \rho)-\W_c(\mu, \nu)\geq \int_Y \varphi^c \dd(\rho-\nu)\;.
\end{align*}  
By convexity of $\E$ and~\eqref{vgrad}, we obtain
\[\E[\rho]-\E[\nu] \geq \int_Y \V[\nu] \dd(\rho-\nu)\]
hence, finally using~\eqref{eqmk2} and the fact that $\rho$ is absolutely continuous with respect to $m_0$, we get
\[\J_\mu[\rho]-\J_\mu[\nu]\geq \int_Y ( \varphi^c+\V[\nu]) \dd(\rho-\nu)\geq 0\]
which means that $\nu$ solves~\eqref{varequil}.

%%%%%%%%%%%%%%%%%%%%%%%%%%%%%
\subsection{Proof of Lemma~\ref{conseqinada}}\label{aconseqinada}
The existence of a minimiser is similar to the proof of Corollary~\ref{coroexemple} since the coercivity of $F$ ensures that minimising sequences are  uniformly integrable and its convexity guarantees sequential weak lower semi continuity of $\J_\mu$.

Let $\nu$ solve~\eqref{varequil} and let us prove that it is bounded away from zero. Let us assume by contraction that $m_0(\{\nu \leq \lambda\})>0$  for every $\lambda>0$. Let $\delta_0>0$, $\delta \in (0, \delta_0)$ (to be chosen later on). Let $A:=\{x\,:\, \delta_0 \le \nu(x) \le M\}$ where $M>0$ is large enough so that $m_0(A)>0$. For small $\eps>0$ such that $\eps <\delta m_0(A)/2$ then define
\[\nu_\eps:=\nu+\eps (u_{A_\delta} -u_A)\]
where $A_\delta:=\{\nu\leq \delta\}$ and for $m_0(B)>0$, $u_B$ denotes the (sort of uniform probability on $B$) $u_B:=m_0(B)^{-1} \un_B$.  Since $\eps <\delta m_0(A)/2$ and $\nu>\delta$ on $A$, $\nu_\eps$ is a probability measure. By optimality of $\nu$, we then have 
\begin{equation}\label{nueppluscher}
0\leq \J_\mu[\nu_\eps]-\J_\mu[\nu]=\W_c(\mu, \nu_\eps)-\W_c(\mu, \nu)+\E[\nu_\eps]-\E[\nu].
\end{equation}
Denoting by $\varphi_\eps$ a Kantorovich potential between $\mu$ and $\nu_\eps$, we first have
\begin{equation*}
  \W_c(\mu, \nu_\eps)-\W_c(\mu, \nu)\leq \eps \int_{Y} \varphi_\eps^c \dd(u_{A_\delta} -u_A)\;.
\end{equation*}
And since $\varphi_\eps^c$ has a modulus of continuity that is uniform with respect to $\eps$ (that of $c$) and  $\varphi_\eps^c$ can be normalised so as to vanish at the same point, $\varphi_\eps^c$ is uniformly bounded independently of $\eps$ and $\delta$. So that $\W_c(\mu, \nu)-\W_c(\mu, \nu_\eps)\leq C_1\eps$ for some constant $C_1$. In a similar way, one finds a constant $C_2$ such that for $\eps$ small enough and uniformly in $\delta$ one has
\[\frac{1}{2}\iint_{Y^2} \phi(y,z)  \dd\nu_\eps(y)\dd(\nu_\eps- \nu)(z)\leq C_2 \eps.\]
Now it remains to estimate the last term namely
\[\begin{split}
\int_{Y} \left[F(\nu_\eps)-F(\nu)\right] \dd m_0=\int_{A_\delta} \left[F(\nu+\eps m_0(A_\delta)^{-1})-F(\nu) \right]\dd m_0\\
+\int_A \left[F(\nu-\eps m_0(A)^{-1})-F(\nu)\right]\dd m_0
\end{split}\]
 since $F$ is Lipschitz on $[\delta_0/2, M]$ the second term can be bounded from above by $C_3 \eps$  for a constant $C_3$ again independent of $\delta$ and $\eps$.  Now let $C:=C_1+C_2+C_3$ thanks to Inada's condition there is some $\delta_1<\delta_0/2$ such that $f \leq -C-1$ on $(0, 2\delta_1]$. Choosing $\delta\leq \delta_1$ and $\eps$ small enough so that $\eps m_0(A_\delta)^{-1}\leq \delta_1$, we then have 
\[\int_{A_\delta} \left[F(\nu+\eps m_0(A_\delta)^{-1})-F(\nu) \right]\dd m_0\leq (-C-1)\eps.\]
Putting everything together, the latter inequality gives the desired contradiction to~\eqref{nueppluscher}.

The proof of the upper bound is similar: one assumes that $\nu\notin \L^{\infty}(m_0)$ and then considers a perturbation of the form  $\nu_\eps:=\nu+\eps (u_{C} -u_{C_M})$ with $C_M:=\{\nu\geq M\}$, $M$ large and $C$ well chosen, the computations are the same as before and the contradiction comes from the Inada condition at $+\infty$: $F'(\nu)=f(\nu)\to+\infty$ as $\nu\to +\infty$. 
%%%%%%%%%%%%%%%%%%%%%%%%%%%%%
 \subsection{Proof of Theorem~\ref{uniqgeodcon}}\label{proofgeodcon}
 Uniqueness of a minimiser follows directly from the strict convexity along generalised geodesics with base $\mu$ of $\J_\mu$ which follows from McCann's condition~\eqref{geodcon1}, the convexity of $v$ and  $\phi$ and the strict convexity along generalised geodesics with base $\mu$ of $\W_2^2(\mu, .)$. 

 Let us now assume that $\gamma$ is an equilibrium and that $\nu$ is its second marginal, then, for some constant $M$,  we have:
 \begin{equation}\label{chepa}
 f(\nu) +v(y)+\int_{Y} \phi(y,z) \dd \nu(z)+\varphi^c \geq M \mbox{ a.e. with an equality $\nu$-a.e.}
 \end{equation}
Thanks to  Inada's condition and the fact that the right hand side is continuous, this implies $\nu$ is bounded away from zero so that~\eqref{chepa} actually is  an equality (Lebesgue) almost everywhere $\V[\nu]+\varphi^c=M$ and thus $\nu$ satisfies
\begin{equation*}
\nu(y)=f^{-1}\left(M-v(y)-\int_{Y} \phi(y,z) \dd \nu(z)-\varphi^c(y)\right)
\end{equation*}
and is therefore continuous.

Let us now prove that $\nu$ solves~\eqref{varequil2}, let $\rho$ be another probability measure (which we can assume to have a positive and continuous density as well), and let $t\in [0,1]\to \nu_t$ denote the generalised geodesic with base $\mu$ joining $\nu$ and $\rho$, {\it i.e.} $\nu_t={T_t}_\# \mu:=((1-t)T_0+t T_1)_\#\mu$ where $T_0$ (resp. $T_1$) denotes the Brenier map between $\mu$ and $\nu$ (resp. $\rho$). Since $(T_0, T_0+t(T_1-T_0))_\#\mu$ has marginals $\nu$ and $\nu_t$ we have
 \begin{equation}\label{w1ent}
 \W_1(\nu, \nu_t)\leq \int_Y \vert T_t-T_0 \vert \dd\mu \leq t \diam(Y)\;.
 \end{equation}
  By the convexity of $\J_\mu$ along generalised geodesics with base $\mu$, setting $g(t)=\J_\mu[\nu_t]$ and using $g(t)\leq (1-t) g(0)+tg(1)$ for all $t\in(0,1)$ we have:
 \[\J_\mu[\rho]-\J_\mu[\nu]=g(1)-g(0)\geq  \frac{1}{t} [g(t)-g(0)]= \frac{1}{t} (\J_\mu[\nu_t]-\J_\mu[\nu]).\]
 Let us write $\nu_t$ as $\nu_t:=\nu+th_t$, by the (usual) convexity of $F$ and that of $\W_2^2(\mu, .)$, we first have
 \[
 \frac{1}{t} \left( \frac{1}{2} \W_2^2(\mu, \nu_t)+\int_Y F(\nu_t)\dd m_0 -  \frac{1}{2} \W_2^2(\mu, \nu)-\int_Y F(\nu)\dd m_0\right) \geq \int_Y [f(\nu) +\varphi^c] h_t \dd m_0
 \]
 Let us now expand $\iint_{Y^2} \phi(y,z) \dd \nu_t(y)\dd \nu_t(z)$ in powers of $t$ as
 \begin{equation*}
 \frac{1}{2t}  \iint_{Y^2} \phi(y,z) \dd \left[(\nu+th_t)(y)(\nu+th_t)(z)- \nu(y)\nu(z)\right]=\iint_{Y^2} \phi(y,z) \dd h_t(y)\dd \nu(z)  +R_t
 \end{equation*}
 where
  \[2\vert R_t\vert =\frac{1}{t} \Big\vert \iint \phi(y,z) \dd(\nu_t-\nu)(y)\dd(\nu_t-\nu)(z)\Big\vert \leq \frac{1}{t} \W_1(\nu_t,\nu) \times \Lip(\psi_t,Y)\leq \diam(Y) \Lip(\psi_t,Y)\]
 where the last inequality follows from~\eqref{w1ent} and $\psi_t$ is defined by 
 \[
 \psi_t(y):=\int_{Y} \phi(y,z) \dd(\nu_t-\nu)(z)
 =\int_{Y} (\phi(y, T_t(x))-\phi(y, T_0(x))) \dd\mu(x) 
 \;.\]
 Since $\nabla \phi$ is locally Lipschitz and $T_t-T_0$ is uniformly bounded by $t\diam(Y)$ we find  that  there is a constant $C$ such that $\Lip(\psi_t, Y) \leq C t$ so that  $R_t= O(t)$. Putting everything together and using ~\eqref{chepa}  we get
 \[\begin{split}
 \J_\mu[\rho]-\J_\mu[\nu]\geq \int_Y (f(\nu(y))+v(y)+\int_Y \phi(y,z) \dd \nu(z)+\varphi^c(y))\dd h_t(y) +R_t\\
 =   \int_{Y} (\V[\nu]+\varphi^c) \dd h_t+R_t\geq  M \int_Y \dd h_t+ R_t=R_t\to 0 \mbox{ as $t\to 0^+$}\end{split}\]
 which proves that $\nu$ is a minimiser.

%%%%%%%%%%%%%%%%%%%%%%%%%%%%%%%%%%%%%%%%
%%%%%%%%%%%%%%%%%%%%%%%%%%%%%%%%%%%%%%%%
%%%%%%%%%%%%%%%%%%%%%%%%%%%%%%%%%%%%%%%%
%%%%%%%%%%%%%%%%%%%%%%%%%%%%%%%%%%%%%%%%
%%%%%%%%%%%%%%%%%%%%%%%%%%%%%%%%%%%%%%%%
%%%%%%%%%%%%%%%%%%%%%%%%%%%%%%%%%%%%%%%%
%%%%%%%%%%%%%%%%%%%%%%%%%%%%%%%%%%%%%%%%
\bibliographystyle{plainnat}
\bibliography{biblio}
\end{document}